\documentclass[11pt]{article}

\usepackage{framed}

\usepackage{fullpage,comment,authblk,stackrel}
\usepackage[small]{caption}
\usepackage{tikz-cd}

\usepackage{nicefrac}
\usepackage{hyperref}
\hypersetup{
	colorlinks   = true,
	citecolor    = red,
        linkcolor    = blue
}
\usepackage{amssymb,mathtools,amsthm, amsmath}
\usepackage{extpfeil}%for xtwoheadleftarrow
\usepackage[all,cmtip]{xy}
\usepackage{eulervm,xfrac,mathdots}
\usepackage{rotating, setspace}

\usepackage{subfig}

\usepackage[capitalize,noabbrev,nameinlink]{cleveref}
\usepackage{bm}

\newtheoremstyle{amit}% name
{7pt}% Space above
{7pt}% Space below
{}% Body font
{7pt}% Indent amount
{\bf}% Theorem head font
{:}% Punctuation after theorem head
{.5em}% Space after theorem head
{}% Theorem head spec (can be left empty, meaning `normal')

\theoremstyle{amit}
\newtheorem{theorem}{Theorem}[subsection]
\newtheorem{definition}[theorem]{Definition}
\newtheorem{lemma}[theorem]{Lemma}
\newtheorem{prop}[theorem]{Proposition}
\newtheorem{corollary}[theorem]{Corollary}

\newtheorem{remark}[theorem]{Remark}
\newtheorem{ex}[theorem]{Example}

\newcommand{\Ffunc}		{\mathsf{F}} % Filtration
\newcommand{\Gfunc}		{\mathsf{G}} % Filtration

\newcommand{\Zfunc}		{\mathsf{Z}} % cycles
\newcommand{\Bfunc}		{\mathsf{B}} % boundaries
\newcommand{\ZB}		{\mathsf{ZB}} % birth-death functions

\newcommand{\Fil}		{\mathsf{Fil}}
\newcommand{\wfil}		{\mathsf{W\textnormal{-}Fil}}
\newcommand{\Dgm}	    {\mathsf{Dgm}}
\newcommand{\wDgm}	    {\mathsf{W\textnormal{-}Dgm}}

\newcommand{\cmet}	    {\mathsf{CatMet}}
\newcommand{\cmm}	    {\mathsf{CatMM}}

\newcommand{\PD}	    {\mathsf{PD}}
\newcommand{\flip}	    {\mathsf{flip}}
\newcommand{\cost}	    {\mathsf{cost}}
\newcommand{\dis}	    {\mathsf{dis}}
\newcommand{\defo}	    {\mathsf{dpl}} %deformation
\newcommand{\supp}      {\mathrm{supp}}
\newcommand{\subcx}          	{\mathsf{SubCx}} %subcomplexes
\newcommand{\dgh}        {d_\mathrm{GH}} %d_GH
\newcommand{\cpl}          	{\mathsf{Cpl}} %couplings
\newcommand{\diag}          	{\mathsf{diag}} %couplings
\newcommand{\defcost}          	{\mathsf{dpl\textnormal{-}cost}} %displacement cost
\newcommand{\len}	    {\mathsf{len}} %length

\newcommand{\VR}	    {\mathsf{VR}}
\newcommand{\WVR}	    {w\mathsf{\textnormal{-}VR}}

\newcommand{\dgw}[1]{d_\mathrm{GW_{#1}}}

\DeclareMathOperator{\ima}{im}

\newcommand{\RR}{\mathbb{R}}

\newcommand{\ZZ}{\mathbb{Z}}
\newcommand{\NN}{\mathbb{N}}
\newcommand{\eps}{\varepsilon} %epsilon

\title{$\ell^p$-Stability of Weighted Persistence Diagrams}
\author[1]{Aziz Burak G\"ulen\footnote{\href{mailto:aziz.burak.guelen@duke.edu}{aziz.burak.guelen@duke.edu}}}
\author[2]{Facundo M\'emoli\footnote{\href{mailto:facundo.memoli@gmail.com}{facundo.memoli@gmail.com}}}
\author[3]{Amit Patel\footnote{\href{mailto:amit.patel@colostate.edu}{amit.patel@colostate.edu}}}
\affil[1]{Department of Mathematics, Duke University}

\affil[2]{Department of Mathematics, Rutgers University}

\affil[3]{Department of Mathematics, Colorado State University}
\date{}

\allowdisplaybreaks

\begin{document}  

\maketitle

\begin{abstract}
   We introduce the concept of weighted persistence diagrams and develop a functorial pipeline for constructing them from finite metric measure spaces. This builds upon an existing functorial framework for generating classical persistence diagrams from finite pseudo-metric spaces. To quantify differences between weighted persistence diagrams, we define the $p$-edit distance for $p \in [1, \infty]$, and—focusing on the weighted Vietoris–Rips filtration—we establish that these diagrams are stable with respect to the $p$-Gromov–Wasserstein distance as a direct consequence of functoriality. In addition, we present an Optimal Transport-inspired formulation of the $p$-edit distance, enhancing its conceptual clarity. Finally, we explore the discriminative power of weighted persistence diagrams, demonstrating advantages over their unweighted counterparts.
\end{abstract}

\tableofcontents

%%%%%%%%%%%%%%%%%%%%%%%%%%%%%%%%%%%%%%%%%%%%%%%%%%
%%%%%%%%%%%%%%%%%%%%%%%%%%%%%%%%%%%%%%%%%%%%%%%%%%

\section{Introduction}
\label{sec:introduction}
Persistent homology is one of the central subjects in Topological Data Analysis (TDA). When analyzing point cloud data, i.e., a finite metric space, a topological summary of the data can be obtained via  persistent diagrams~\cite{edelsbrunner2000, zomorodian2005, burago2001} of a filtration associated to the dataset. When analyzing finite metric spaces, the initial step involves constructing a filtration from the metric space. Two common approaches for achieving this are the Vietoris-Rips filtration and the \v{C}ech filtration. 

A fundamental result in persistent homology is that persistent diagrams are stable with respect to the perturbations of input filtration / metric space. Classical stability results for persistence diagrams typically rely on the Gromov–Hausdorff distance between input metric spaces and the bottleneck distance between output persistence diagrams~\cite{cohen-steiner2007, chazal2009}. In this work, we study the $p$-stability of weighted persistence diagrams—that is, classical persistence diagrams equipped with probability measures. To this end, we replace the input collection of metric spaces with the collection of metric measure spaces (or mm-spaces), which are metric spaces endowed with fully supported probability measures. On the output side, we extend the classical notions of filtrations and persistence diagrams to their weighted counterparts.

Metric measure spaces are triples $(X, d_X, \mu_X)$ that combine geometric structure (a metric $d_X$ on $X$) with probabilistic structure (a probability measure $\mu_X$ on $X$), providing a powerful framework for comparing and analyzing complex data. They play a central role in shape analysis, topological data analysis (TDA) \cite{chazal2009}, and optimal transport \cite{memoli2007use,facundo-gw,sturm-geodesic}. Applications include comparing point clouds or graphs via the \emph{Gromov–Wasserstein distance} \cite{facundo-gw}, modeling variability in anatomical shapes, and studying empirical distributions in geometric inference. The mm-space perspective enables coordinate-free, geometry-aware comparisons in high-dimensional and structured data settings. One of the concepts pertaining to mm-spaces is the notion of \emph{distance distribution} \cite{osada2002shape,facundo-gw,memoli-needham} of a mm-space $(X,d_X,\mu_X)$ which is the probability measure on the real line arising as the pushforward of $\mu_X\otimes \mu_X$ via $d_X$. Besides being stable under the Gromov-Wasserstein distance, their direct comparison yields fast methods for object matching.

Motivated by the flexibility inherent to the mm-space representation of data, we introduce suitable notions of weighted filtrations and weighted persistence diagrams. Our central approach is to construct a functorial pipeline that maps mm-spaces to weighted persistence diagrams, and to analyze this construction using (different instantiations of) the notion of edit distance. When appropriately designed, such functorial pipelines yield stability results with respect to edit distances, as shown in~\cite{mccleary2022edit, orthogonal-mobius}. We adopt this framework to establish the $p$-stability of weighted persistence diagrams. To capture variability in both geometric and probabilistic aspects, we introduce weights on filtrations and persistence diagrams, allowing us to define $p$-edit distances between them. Our main stability result expresses the $p$-edit distance between weighted persistence diagrams in terms of the $p$-Gromov–Wasserstein distance between the input mm-spaces. To further enhance the interpretability of this framework, we show that the $p$-edit distance admits an Optimal Transport-like formulation, shedding light on its underlying structure and discriminative behavior.

%%%%%%%%%%%%%%%%%%%%%%%%%%%%%%%%%%%%%%%%%%%%%%%%%%
%%%%%%%%%%%%%%%%%%%%%%%%%%%%%%%%%%%%%%%%%%%%%%%%%%

%%%%%%%%%%%%%%%%%%%%%%%%%%%%%%%%%%%%%%%%%%%%%%%%%%
\subsection{Previous Work}
%%%%%%%%%%%%%%%%%%%%%%%%%%%%%%%%%%%%%%%%%%%%%%%%%%

In~\cite{mccleary2022edit}, the authors introduce a functorial pipeline  
\[
\Fil(K) \to \Dgm
\]  
for producing persistence diagrams from filtrations of a fixed finite simplicial complex $K$. They organize the collection of filtrations into a category $\Fil(K)$ and the collection of persistence diagrams into a category $\Dgm$. Within this framework, the degree-$d$ persistence diagram of a filtration is obtained via M\"obius inversion---a combinatorial technique for compressing and summarizing mathematical data. Specifically, M\"obius inversion is applied to the degree-$d$ birth-death function associated with the filtration. The functoriality of the pipeline is a consequence of Rota's Galois Connection Theorem (RGCT)~\cite{rota64, greenemobius, mccleary2022edit, gal-conn}, which links the morphisms in $\Fil(K)$ and $\Dgm$---both defined in terms of Galois connections---to the application of M\"obius inversion. RGCT thus ensures that the assignment of persistence diagrams is functorial.

Each morphism in the categories $\Fil(K)$ and $\Dgm$ is assigned a non-negative cost. These cost assignments induce metrics—referred to as edit distances—on the objects of $\Fil(K)$ and $\Dgm$, denoted $d_{\Fil(K)}^E$ and $d_{\Dgm}^E$, respectively. Leveraging the functoriality of the pipeline, the authors establish the following stability result.

\begin{theorem}[{\cite[Theorem 8.4]{mccleary2022edit}}]
    Let $\Ffunc$ and $\Gfunc$ be two filtrations in $\Fil(K)$ and let $\PD_d^\Ffunc$ and $\PD_d^\Gfunc$ denote their degree-$d$ persistence diagrams respectively. Then,
    \[
    d_{\Dgm}^E (\PD_d^\Ffunc, \PD_d^\Gfunc) \leq d_{\Fil(K)}^E (\Ffunc, \Gfunc).
    \]
\end{theorem}

It is important to emphasize that this stability result does not concern metric spaces as inputs; rather, it addresses filtrations as the primary objects of study. Specifically, it establishes the stability of persistence diagrams with respect to perturbations of filtrations in terms of edit distance.

We also note that related $p$-stability results for (unweighted) persistence diagrams have been established in~\cite{LpCohenSteiner2010, skraba-wass}. Like the result above, these too do not take metric (or metric measure) spaces as input. Instead, they consider functions and the filtrations derived from their sublevel sets.

%%%%%%%%%%%%%%%%%%%%%%%%%%%%%%%%%%%%%%%%%%%%%%%%%%
\subsection{Our Contributions}
%%%%%%%%%%%%%%%%%%%%%%%%%%%%%%%%%%%%%%%%%%%%%%%%%%

First, we extend the functorial pipeline of~\cite{mccleary2022edit} to  
\[
\cmet \to \Fil \to \Dgm,
\]
where $\cmet$ denotes the category of finite pseudo-metric spaces and $\Fil$ denotes the category of one-parameter filtrations of finite simplicial complexes. By carefully designing the morphisms in $\cmet$, we realize the Vietoris--Rips filtration as a functor  
\[
\VR : \cmet \to \Fil.
\]
In addition, we equip each morphism in $\cmet$ with a non-negative cost, thereby inducing a notion of edit distance $d_{\cmet}^E$ between metric spaces. One of our main results is the following.

\begin{theorem}\label{thm: classical edit distance stability of pd}
    Let $(X,d_X)$ and $(Y,d_Y)$ be finite pseudo-metric spaces and let $\VR(X)$ and $\VR(Y)$ denote their Vietoris-Rips filtrations respectively. Then,
    \[
    d_{\Dgm}^E (\PD_d^{\VR(X)}, \PD_d^{\VR(Y)} ) \leq 2 \cdot d_{\cmet}^E ((X,d_X), (Y,d_Y)) = 4 \cdot \dgh ((X,d_X), (Y,d_Y)).
    \]
\end{theorem}

The inequality above follows from the functoriality of both $\VR$ and $\PD_d$ (see~\cref{prop: unweighted vr is functor} and \cref{prop: functoriality of pd}). The Gromov–Hausdorff distance $\dgh$ quantifies how far two finite (pseudo-)metric spaces are from being isometric. The identity $d_{\cmet}^E = 2\dgh$ (see~\cref{prop: gh as edit}) arises directly from the specific construction of the category $\cmet$ and the associated definition of morphism costs.

\begin{figure}
\centering
\includegraphics[width=0.5\textwidth]{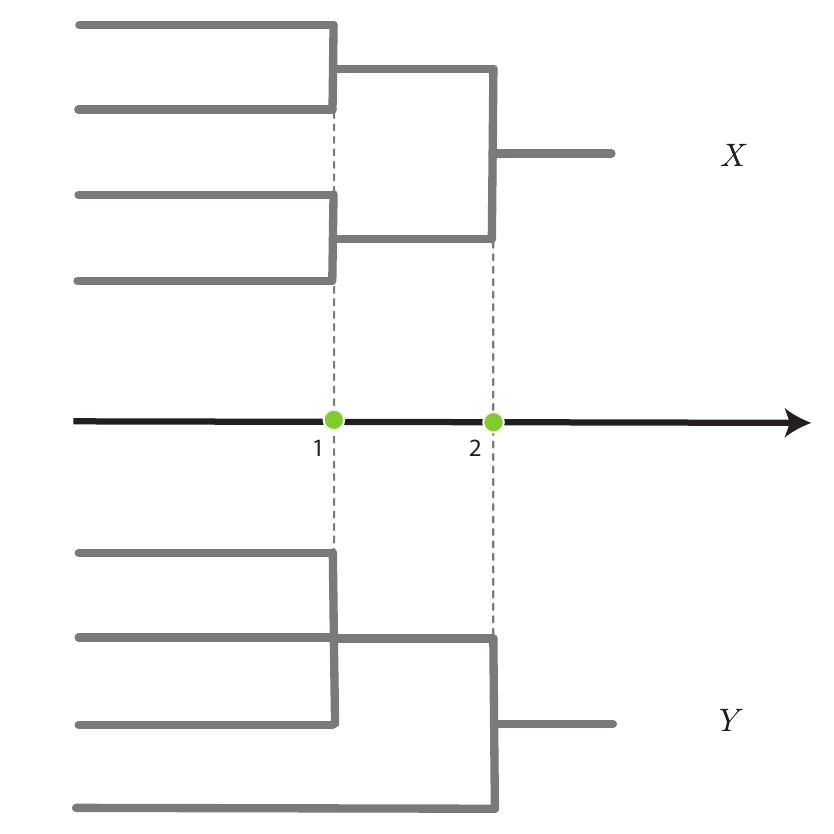}
\caption{Two ultrametric spaces $X$ and $Y$, each with 4 points, shown in the figure (through their dendrogram representations) have the same Vietoris-Rips persistence diagrams in all degrees. However, when regarded as mm-spaces (by endowing them with their respective uniform measures), their weighted Vietoris-Rips persistence diagrams are different; see \Cref{ex:ums}.}
\label{fig:ums}
\end{figure}

Second, we introduce the notions of weighted filtrations (\cref{defn: weighted filtration}) and weighted persistence diagrams (\cref{defn: weighted persistence diagrams}), extending their classical unweighted counterparts. Furthermore, we construct a functorial pipeline for obtaining weighted persistence diagrams from metric measure spaces (mm-spaces), summarized as follows:
\[
\cmm \to \wfil \to \wDgm.
\]
Here, $\cmm$ denotes the category of finite mm-spaces, $\wfil$ the category of weighted filtrations, and $\wDgm$ the category of weighted persistence diagrams. We define the weighted Vietoris--Rips filtration, denoted $\WVR$, and show that it defines a functor:
\[
\WVR : \cmm \to \wfil.
\]

Weighted Vietoris-Rips persistence diagrams capture more information than their standard counterparts. For example the two ultrametric spaces $X$ and $Y$ from \Cref{fig:ums} have the same persistence diagrams in all degrees whereas their weighted persistence diagrams are different; see \Cref{ex:ums} for details.

We assign a non-negative $p$-cost (for $p \in [1,\infty]$) to each morphism in the categories appearing in the pipeline, thereby defining corresponding $p$-edit distances, denoted $d_{\cmm}^{E,p}$, $d_{\wfil}^{E,p}$, and $d_{\wDgm}^{E,p}$. Our main result establishes the $p$-stability of weighted persistence diagrams with respect to these distances.

\begin{theorem}\label{thm: p stability of weighted persistence diagrams main thm} Let $(X,d_X,\mu_X)$ and $(Y,d_Y,\mu_Y)$ be two finite mm-spaces. Let $\WVR (X)$ and $\WVR(Y)$ denote their weighted Vietoris-Rips filtrations and let $w\textnormal{-}\PD_d^\VR(X)$ and $w\textnormal{-}\PD_d^\VR(Y)$ denote the weighted degree-$d$ persistence diagrams of $\WVR (X)$ and $\WVR(Y)$. Then, \begin{align*} d_{\wDgm}^{E,p} (w\textnormal{-}\PD_d^\VR(X), w\textnormal{-}\PD_d^\VR(Y)) &\leq 2\cdot 4^\frac{1}{p} \cdot d_{\cmm}^{E,p} ((X,d_X,\mu_X)), (Y,d_Y,\mu_Y)) \\ &= 4^\frac{p+1}{p} \cdot \dgw{p} ((X,d_X,\mu_X)), (Y,d_Y,\mu_Y)) \end{align*} \end{theorem}

The inequality above follows from the functoriality of the pipeline; see~\cref{prop: weighted vr is functor} and~\cref{prop: weighted pd are functorial}. The $p$-Gromov–Wasserstein distance, denoted $\dgw{p}$, quantifies how far two mm-spaces are from being isomorphic. The identity $d_{\cmm}^{E,p} = 2 \cdot \dgw{p}$ (see~\cref{prop: p-gw as edit}) arises—analogously to the case of $\dgh$—from the specific construction of the category $\cmm$ and the definition of morphism costs within it.

Although ($p$-)edit distance stability follows directly from functoriality, the precise meaning of the ($p$-)edit distance could remain unclear at first glance. In~\cref{sec: ot-like interpret}, we demonstrate that the ($p$-)edit distance between (weighted) persistence diagrams can be reformulated independently of category theory or the formal definition of edit distance. These alternative formulations—given in~\cref{prop: interpretation unweighted} and~\cref{prop: OT formulation of dEp}—provide an Optimal Transport-like interpretation of the ($p$-)edit distance, thereby enhancing its conceptual clarity and interpretability.

Persistence diagrams are widely used in practice for their ability to compactly encode geometric information, particularly in applications involving shape characterization. In~\cref{sec: discriminating power}, we investigate the discriminative power of \emph{weighted persistence diagrams}, which, in the case of Vietoris-Rips filtrations of finite metric spaces, combine two sources of information: classical persistence diagrams and the global distribution of distances. We show that this synthesis enhances discriminative strength beyond what either component provides on its own. Notably, the $p$-edit distance is most discriminative when $p = \infty$. For finite values of $p$, the distance becomes insensitive to the persistence diagram component and depends only on the weight distribution (\cref{prop: independence from diagram finite p}). Nevertheless, it can still distinguish between weighted persistence diagrams even when the weights, viewed as mm-spaces, are isomorphic—thanks to its sensitivity to the ordering of intervals (see~\cref{ex: same gdd but displacement distinguishes}).

%\end{framed}

%%%%%%%%%%%%%%%%%%%%%%%%%%%%%%%%%%%%%%%%%%%%%%%%%%
\subsection{Acknowledgements}
%%%%%%%%%%%%%%%%%%%%%%%%%%%%%%%%%%%%%%%%%%%%%%%%%% 2310412
This work was partially supported by NSF DMS \#2301359, NSF CCF \#2310412 and NSF RI \#1901360.
%%%%%%%%%%%%%%%%%%%%%%%%%%%%%%%%%%%%%%%%%%%%%%%%%%
%%%%%%%%%%%%%%%%%%%%%%%%%%%%%%%%%%%%%%%%%%%%%%%%%%
\section{Preliminaries}
\label{sec:prelim}
%%%%%%%%%%%%%%%%%%%%%%%%%%%%%%%%%%%%%%%%%%%%%%%%%%
%%%%%%%%%%%%%%%%%%%%%%%%%%%%%%%%%%%%%%%%%%%%%%%%%%

In this section, we recall the background concepts that will be used in the rest of the paper.

%%%%%%%%%%%%%%%%%%%%%%%%%%%%%%%%%%%%%%%%%%%%%%%%%%
\subsection{Posets and Galois Connections}\label{sec:lattices}
%%%%%%%%%%%%%%%%%%%%%%%%%%%%%%%%%%%%%%%%%%%%%%%%%%

A \emph{partially ordered set} (poset) is a set $X$ with a 
reflexive, antisymmetric, and transitive relation $\leq$.
A \emph{Galois connection} between two posets $P$ and $Q$ is a pair, $(f,g)$, of order-preserving maps $f : P \to Q$ and $g : Q \to P$ such that
\[
f(p) \leq q \iff p \leq g(q)
\]
for every $p \in P$ and $q\in Q$. We refer to $f$ as the~\emph{left adjoint} and refer to $g$ as the~\emph{right adjoint}. We will also use the notation $f : P \leftrightarrows Q : g$ to denote a Galois connection.

We say that an order-preserving map $f:P \to Q$ has a right adjoint (or, admits a right adjoint) if there is an order-preserving map $g : Q\to P$ such that the pair $(f,g)$ is a Galois connection. Note that, if an order-preserving map $f: P \to Q$ has a right adjoint $g$, then, it must be unique and given by
\[
g(q) := \max \{ p\in P \mid f(p) \leq q\}.
\]

\paragraph{M\"obius Inversion.} Let $P$ be a poset and $\mathcal{G}$ be an abelian group. Let $m : P \to \mathcal{G}$ be any function. The \emph{M\"obius inverse} of $m$ is defined to be the unique function $\partial_P m : P \to \mathcal{G}$, if it exists, satisfying
\[
m (p) = \sum_{p' \leq p} \partial_P m (p')
\]
for all $p \in P$.

\begin{definition}[Pushforward and pullback]
    Let $f: P \to Q$ be any map between two sets, and let $m : P \to \mathcal{G}$ be any function. The~\emph{pushforward} of $m$ along $f$ is the function $f_\sharp m : Q \to \mathcal{G}$ given by
    $$f_\sharp m (q) := \sum_{p \in f^{-1}(q)} m(p).$$
    Let $h : Q \to \mathcal{G}$ be any function. The~\emph{pullback} of $h$ along $f$ is the function $f^\sharp h : P \to \mathcal{G}$ given by
$$(f^\sharp h )(p) := h( f(p)).$$
    
\end{definition}

\begin{theorem}[Rota's Galois Connection Theorem (RGCT)~{\cite[Theorem 3.1]{gal-conn}}]\label{thm:rgct}
    Let $P$ and $Q$ be finite posets and $f : P \leftrightarrows Q : g$ be a Galois connection. Then,
    \begin{equation}\label{eqn: rgct push vs pull}
        (f)_\sharp \circ \partial_P = \partial_Q \circ (g)^\sharp.
    \end{equation}
\end{theorem}

%%%%%%%%%%%%%%%%%%%%%%%%%%%%%%%%%%%%%%%%%%%%%%%%%%
\subsection{Simplicial Complexes and Filtrations}
%%%%%%%%%%%%%%%%%%%%%%%%%%%%%%%%%%%%%%%%%%%%%%%%%%

An (abstract) \emph{simplicial complex} $K$ over a finite vertex set $V$ is a non-empty collection of non-empty subsets of $V$ with the property that for every $\sigma \in K$, if $\tau\subseteq \sigma$, then $\tau \in K$. An element $\sigma \in K$ is called a \emph{$d$-simplex} if the cardinality of $\sigma$ is $d+1$. An \emph{orientation} on a simplex is an ordering of its vertices up to an even permutation.

An \emph{oriented simplex}, denoted $[\sigma]$, is a simplex $\sigma \in K$ whose vertices are ordered. We always assume that ordering on simplicies is inherited from a pre-determined ordering on $V$. Let $\mathcal{S}_d^K$ denote the set of all oriented $d$-simplicies of $K$.

The \emph{$d$-th chain group of $K$}, denoted $C_d^K$, is the vector space over a fixed field $\Bbbk$ with basis $\mathcal{S}_d^K$. Let $n_d^K := |\mathcal{S}_d^K| = \dim_{\Bbbk} (C_d^K)$. 

The \emph{$d$-th boundary map} $\partial_d^K : C_d^K \to C_{d-1}^K$ is defined by
\[
    \partial_d^K ([v_0,\ldots,v_d]) := \sum_{i=0}^{d}(-1)^i[v_0,\ldots,\hat{v}_i,\ldots,v_d]
\]
for every oriented $d$-simplex $\sigma = [v_0,\ldots,v_d]\in \mathcal{S}_d^K$, where $[v_0,\ldots,\hat{v}_i,\ldots,v_d]$ denotes the omission of the $i$-th vertex, and extended linearly to $C_d^K$. We denote by $\Zfunc_d(K)$ the space of $d$-cycles of $K$, that is $\Zfunc_d(K) := \ker (\partial_d^K)$, and we denote by $\Bfunc_d(K)$ the space of $d$-boundaries of $K$, that is $\Bfunc_d(K):= \ima (\partial_{d+1}^K)$. Additionally, we denote by $H_d(K)$ the $d$-th homology group of $K$, that is $H_d(K) := \Zfunc_d(K) / \Bfunc_d(K)$.

For a finite simplicial complex $K$, let $\subcx(K)$ denote the poset of subcomplexes of $K$, ordered by inclusion. A \emph{simplicial filtration of $K$} is an order-preserving map $\Ffunc : P \to \subcx (K)$, where $P$ is a finite poset with a maximum element $\top_P$ such that $\Ffunc(\top_P) = K$. A \emph{$1$-parameter} filtration is filtration over a finite linearly ordered set $P = \{0 = p_0 < p_1<\cdots<p_n \}\subseteq [0,\infty)$. When $\Ffunc : \{ 0 = p_0 < p_1<\cdots<p_n\} \to \subcx (K)$ is a $1$-paramater filtration, we use the notation $K_i$ for the simplicial complex $\Ffunc(p_i)$ and write $\Ffunc = \{K_i \}_{i=0}^n$ to denote the simplicial filtration. Note that $K_n = K$.

%%%%%%%%%%%%%%%%%%%%%%%%%%%%%%%%%%%%%%%%%%%%%%%%%%
\subsection{Edit Distance}
\label{subsec: edit}
%%%%%%%%%%%%%%%%%%%%%%%%%%%%%%%%%%%%%%%%%%%%%%%%%%
Let $\mathcal{C}$ be a category and assume that for every object $A$ and $B$ in $\mathcal{C}$ and for every morphism $f : A \to B$ from $A$ to $B$, there is a non-negative cost, denoted $\cost_\mathcal{C} (f) \in \RR_+$, associated to the morphism. For two objects $A$ and $B$ in this category, a~\emph{path}, $\mathcal{P}$, is a finite sequence of morphisms
\[
\mathcal{P} : \; \; A \xleftrightarrow{\makebox[1cm]{$f_1$}} D_1 \xleftrightarrow{\makebox[1cm]{$f_2$}} \cdots \xleftrightarrow{\makebox[1cm]{$f_{k-1}$}} D_{k-1} \xleftrightarrow{\makebox[1cm]{$f_k$}} B
\]
where $D_i$s are objects in $\mathcal{C}$ and $\leftrightarrow$ indicates a morphism in either direction. The~\emph{cost of a path} $\mathcal{P}$, denoted $\cost_\mathcal{C} (\mathcal{P})$, is the sum of the cost of all morphisms in the path:
\[
\cost_\mathcal{C} (\mathcal{P}) := \sum_{i=1}^k \cost_\mathcal{C} (f_i).
\]

\begin{definition}[Edit distance]\label{defn: edit dist}
    The~\emph{edit distance} $d_{\mathcal{C}}^E (A,B)$ between two objects $A$ and $B$ in $\mathcal{C}$ is the infimum, over all paths between $A$ and $B$, of the cost of such paths:
    \[
    d_\mathcal{C}^E(A,B) := \inf_{\mathcal{P}} \cost_\mathcal{C}(\mathcal{P}).
    \]
\end{definition}

The following proposition will serve as the main tool for establishing stability results in this paper.

\begin{prop}[Functorial edit distance stability]\label{prop: general edit distance stability}
    Let $\mathcal{C}$ and $\mathcal{D}$ be two categories such that for every morphism in these categories $f$ in $\mathcal{C}$ and $g$ in $\mathcal{D}$, there are costs $\cost_{\mathcal{C}} (f), \cost_{\mathcal{D}}(g) \in \RR_+$ associated to them. Let $\Ffunc : \mathcal{C} \to \mathcal{D}$ be a functor and $M \in (0,\infty)$ be such that for every morphism $f$ in $\mathcal{C}$, it holds that
    \[
    \cost_{\mathcal{D}} (\Ffunc (f)) \leq M \cdot \cost_{\mathcal{C}} (f)
    \]
    Then, for any two object $A$ and $B$ in $\mathcal{C}$, we have
    \[
    d_{\mathcal{D}}^E (\Ffunc(A), \Ffunc(B)) \leq M \cdot d_{\mathcal{C}}^E (A,B).
    \]
\end{prop}

\begin{proof}
    Let $A$ and $B$ be two objects in $\mathcal{C}$ and let 
    \[
    \mathcal{P} : \; \; A \xleftrightarrow{\makebox[1cm]{$f_1$}} D_1 \xleftrightarrow{\makebox[1cm]{$f_2$}} \cdots \xleftrightarrow{\makebox[1cm]{$f_{k-1}$}} D_{k-1} \xleftrightarrow{\makebox[1cm]{$f_k$}} B
    \]
    be a path between $A$ and $B$ in $\mathcal{C}$. Since $\Ffunc$ is functor, we obtain the following path
    \[
    \Ffunc(\mathcal{P}) : \; \; \Ffunc(A) \xleftrightarrow{\makebox[1cm]{$\Ffunc(f_1)$}} \Ffunc(D_1) \xleftrightarrow{\makebox[1cm]{$\Ffunc(f_2)$}} \cdots \xleftrightarrow{\makebox[1cm]{$\Ffunc(f_{k-1})$}} \Ffunc(D_{k-1}) \xleftrightarrow{\makebox[1cm]{$\Ffunc(f_k)$}} \Ffunc(B)
    \]
    between $\Ffunc(A)$ and $\Ffunc(B)$. And, we have that $\cost_{\mathcal{D}}(\Ffunc(f_i)) \leq M \cdot \cost_{\mathcal{C}} (f_i)$ for every $i = 1,2, \ldots, k$. Therefore, we have that
    \[
    \cost_{\mathcal{D}} (\Ffunc (\mathcal{P})) \leq M \cdot \cost_{\mathcal{C}} (\mathcal{P}).
    \]
    Hence, we obtain 
    \[
    d_{\mathcal{D}}^E (\Ffunc(A), \Ffunc(B)) \leq M \cdot d_{\mathcal{C}}^E (A,B).
    \]
\end{proof}

%%%%%%%%%%%%%%%%%%%%%%%%%%%%%%%%%%%%%%%%%%%%%%%%%%
\subsection{Metric Spaces}
%%%%%%%%%%%%%%%%%%%%%%%%%%%%%%%%%%%%%%%%%%%%%%%%%%
A \emph{pseudo-metric space} is a pair $(X, d_X)$, where $X$ is a non-empty set 
and $d_X : X \times X \to \RR_+ \cup \{ \infty \}$ a function
such that, for all $a,b,c \in X$, the following conditions hold:
    \begin{itemize}
        \item $d_X(a,b) = d_X(b,a)$
        \item $d_X(a,c) \leq d_X(a,b) + d_X(b,c)$.
    \end{itemize}
Note that these conditions imply $d_X(a,b) \geq 0$, for all $a,b \in X$. A pseudo-metric space $(X,d_X)$ is called a metric space if $d_X(a,b) = 0 \iff a = b$.
A \emph{$1$-Lipschitz} map from a metric space $(X, d_X)$ to 
a metric space $(Y, d_Y)$
is a function $f : X \to Y$ such that for all $a,b \in X$,
$d_Y \big( f(a), f(b) \big) \leq d_X (a,b)$.
Throughout this paper, unless explicitly mentioned otherwise, we operate under the assumption that (pseudo)-metric spaces are finite.

\begin{definition}[Distortion]
    The \emph{distortion} of a set map $f : X \to Y$ between two (pseudo)-metric spaces $(X,d_X)$ and $(Y,d_Y)$ is
\[
\dis (f) := \max_{x,x' \in X} |d_X(x,x') - d_Y (f(x), f(x'))|.
\]

\end{definition}

\begin{definition}[Product]
    The \emph{product} of two (pseudo)-metric spaces $(X, d_X)$ and $(Y, d_Y)$ is the 
(pseudo)-metric space $\left(X \times Y, d_{X \times Y}\right)$ where
$$d_{X\times Y} \left( (x,y), (x',y') \right) := \max \big\{ d_X(x,x'), d_Y(y,y') \big \}.$$    
\end{definition}

Note that the two projections $\pi_X : (X \times Y, d_{X \times Y}) \to (X, d_X)$ and
$\pi_Y : (X \times Y, d_{X \times Y}) \to (Y, d_Y)$ are $1$-Lipschitz.

\paragraph{Gromov-Hausdorff Distance.}

Let $(X,d_X)$ be a (not necessarily finite) metric space and let $Z_1$ and $Z_2$ be two subsets of $X$. The \emph{Hausdorff distance} between $Z_1$ and $Z_2$ is defined to be

\[
d_{\mathrm{H}}^X (Z_1, Z_2) := \max \left \{  \sup_{z_1 \in Z_1} d_X (z_1, Z_2) , \sup_{z_2 \in Z_2} d_X (Z_1, z_2)  \right \} .
\]

For two metric spaces $(X,d_X)$ and $(Y, d_Y)$ the \emph{Gromov-Hausdorff distance},~\cite{EDWARDS1975121, gromov07}, between them is defined as the infimum $r > 0$ for which there exists a metric space $(Z,d_Z)$ together with two isometric embeddings $\psi_X : (X,d_X) \hookrightarrow (Z,d_Z)$ and $\psi_Y : (Y,d_Y) \to (Z,d_Z)$ such that $d_{\mathrm{H}}^Z (\psi_X (X), \psi_Y(Y)) < r$, that is,

\[
\dgh (X,Y) := \inf_{Z, \psi_X, \psi_Y}  d_{\mathrm{H}}^Z \left(\psi_X (X), \psi_Y(Y)\right) .
\]

A \emph{tripod} between two sets $X$ and $Y$ is a pair of surjections $\phi_X$ and $\phi_Y$ from a set $Z$ to $X$ and $Y$ respectively, cf.~\cite[Section 4.1.1]{memoli17}. We express this data through the diagram
\begin{center}
    \begin{tikzcd}
    \mathfrak R: & X & Z \arrow[l, "\phi_X"', two heads] \arrow[r, "\phi_Y", two heads] & Y
    \end{tikzcd}
\end{center}

For two (not necessarily finite) metric spaces $(X,d_X)$ and $(Y,d_Y)$, the \emph{distortion} of a tripod $\mathfrak{R}$ between $X$ and $Y$ is defined as
\[
\dis (\mathfrak{R}) := \sup_{z_1, z_2 \in Z} \big | d_X (\phi_X(z_1), \phi_X(z_2)) - d_Y (\phi_Y(z_1), \phi_Y(z_2) \big |.
\]

\begin{theorem}[{\cite[Theorem 7.3.25]{burago2001}}]\label{thm: tripod version gh}
    For any two bounded metric spaces $X$ and $Y$,
    \[
    \dgh (X,Y)  =\frac{1}{2} \inf \left\{ \dis(\mathfrak{R}) \mid \mathfrak{R} : X \xtwoheadleftarrow{\phi_X} Z \xtwoheadrightarrow{\phi_y} Y \right \}
    \]
\end{theorem}

Note that, in the original statement of the above theorem in~\cite{burago2001}, $\mathfrak{R}$ is assumed to be a \emph{correspondence}, a notion that can be substituted with tripods, as elucidated in~\cite[Example 7.3.19, Exercise 7.3.20]{burago2001}. Hence, we express it in terms of tripods here. Note also that~\cref{thm: tripod version gh} can be utilized to define the notion of Gromov-Hausdorff distance between finite pseudo-metric spaces as follows.

\begin{definition}[Gromov-Hausdorff distance]
    Let $(X,d_X)$ and $(Y,d_Y)$ be two finite pseudo-metric spaces. Then, the \emph{Gromov-Hausdorff distance} between $(X,d_X)$ and $(Y,d_Y)$ is defined as
    \[
    \dgh(X,Y) := \frac{1}{2} \inf \left\{ \dis(\mathfrak{R}) \mid \mathfrak{R} : X \xtwoheadleftarrow{\phi_X} Z \xtwoheadrightarrow{\phi_y} Y \right \},
    \]
    where 
    \[
    \dis (\mathfrak{R}) := \max_{z_1, z_2 \in Z} \big | d_X (\phi_X(z_1), \phi_X(z_2)) - d_Y (\phi_Y(z_1), \phi_Y(z_2) \big |.
    \]
\end{definition}

%%%%%%%%%%%%%%%%%%%%%%%%%%%%%%%%%%%%%%%%%%%%%%%%%%
\subsection{Measure Spaces}
%%%%%%%%%%%%%%%%%%%%%%%%%%%%%%%%%%%%%%%%%%%%%%%%%%

A \emph{measurable space} is a pair $(X, \Sigma)$, where $X$ is a set and $\Sigma$
a set of subsets of $X$ called a $\sigma$-algebra satisfying the following conditions:
    \begin{itemize}
        \item $X \in \Sigma$
        \item if $A \in \Sigma$, then its complement $X \setminus A$
        is in $\Sigma$
        \item the union of countably many elements of $\Sigma$ is an
        element of $\Sigma$.
    \end{itemize}
If $X$ is a topological space, the smallest $\sigma$-algebra
generated by its set of open sets is called the \emph{$\sigma$-algebra
of Borel sets}.
A \emph{measurable function} from a measurable space $(X_1, \Sigma_1)$ to a measurable
space $(X_2, \Sigma_2)$, written $f : (X_1, \Sigma_1) \to (X_2, \Sigma_2)$, 
is a function $f : X_1 \to X_2$ such that for every $A_2 \in \Sigma_2$,
$f^{-1}(A_2) \in \Sigma_1$.

A \emph{measure} on a measurable space $(X, \Sigma)$ is a
function $\mu : \Sigma \to \RR \cup \{ \infty \}$
satisfying the following conditions:
    \begin{itemize}
        \item for all $A \in \Sigma$,  $\mu(A) \geq 0$
        \item $\mu(\emptyset) = 0$
        \item for all pairwise disjoint countable family of subsets $\{ A_i \}_{i = 1}^{\infty}$ in $\Sigma$, 
        $\mu \big( \bigcup_{i = 1}^\infty A_i \big) = \sum_{i=1}^\infty \mu(A_i)$.
    \end{itemize}
A \emph{probability measure} is any measure with total mass $\mu(X)$ equal to $1$.

Given a measureable function $f : (X_1, \Sigma_1) \to (X_2, \Sigma_2)$ and a measure
$\mu_1$ on $(X_1, \Sigma_1)$, the \emph{pushforward measure} 
$f_\# \mu_1$ on $(X_2, \Sigma_2)$ is 
$f_\# \mu_1 (A_2) := \mu_1 \big( f^{-1}(A_2) \big)$, for all $A_2 \in \Sigma_2$.

\begin{definition}[Product of measurable spaces]
    The \emph{product} of two measurable spaces $(X_1, \Sigma_1)$ and $(X_2, \Sigma_2)$ is $(X_1 \times X_2, \Sigma_1 \otimes \Sigma_2)$, where $\Sigma_1 \otimes \Sigma_2$ is the smallest $\sigma$-algebra containing the sets of the form $A_1 \times A_2$ for $A_1 \in \Sigma_1$ and $A_2 \in \Sigma_2$.
\end{definition}

Note that the two projections $\pi_1 : (X_1 \times X_2, \Sigma_1 \otimes \Sigma_2)
\to (X_1, \Sigma_1)$ and 
$\pi_2 : (X_1 \times X_2, \Sigma_1 \times \Sigma_2) \to (X_2, \Sigma_2)$
are measurable functions.

\begin{definition}[Product measure]
    Given a measure $\mu_1$ on $(X_1, \Sigma_1)$ and a measure $\mu_2$ on $(X_2, \Sigma_2)$,
the \emph{product measure} on
$(X_1 \times X_2, \Sigma_1 \otimes \Sigma_2)$
is $$(\mu_1 \otimes \mu_2) (A_1 \times A_2) := \mu_1(A_1) \cdot \mu_2(A_1),$$ for all $A_1 \times A_2 \in \Sigma_1 \times \Sigma_2$.
We have ${\pi_1}_{\#} (\mu_1 \otimes \mu_2) = \mu_1$ and
${\pi_2}_{\#} (\mu_1 \otimes \mu_2) = \mu_2$.
\end{definition}

%%%%%%%%%%%%%%%%%%%%%%%%%%%%%%%%%%%%%%%%%%%%%%%%%%
\subsection{Metric Measure Spaces}
\label{subsec: mm-spaces}
%%%%%%%%%%%%%%%%%%%%%%%%%%%%%%%%%%%%%%%%%%%%%%%%%%

The metric in a metric space $(X, d_X)$ induces a topology on $X$.
A \emph{metric (probability) measure space}, or \emph{mm-space} for short, is a triple $(X, d_X, \mu_X)$,
where $(X, d_X)$ is a metric space and $\mu_X$ is a probability measure
on its $\sigma$-algebra of Borel sets, denoted $B (d_X)$, with full support, i.e., $$\supp(\mu_X) := \overline{ \{ x \in X \mid \text{for every open $N_x\subseteq X$ } (x\in N_x \implies \mu_X(N_x)>0)\}  }  =  X$$ Throughout this paper, unless explicitly stated otherwise, we will assume that mm-spaces are finite.

\begin{definition}[Monge maps]
    A function $f: X \to Y$ from an mm-space $(X,d_X, \mu_X)$ to another mm-space $(Y,d_Y,\mu_Y)$ is called a \emph{Monge} map if $f_\# \mu_X = \mu_Y$. 
\end{definition}

\begin{definition}[Lipschitz mm-space function]
    A \emph{Lipschitz mm-space function} between two (not necessarily finite) mm-spaces $(X, d_X, \mu_X)$ and $(Y, d_Y, \mu_Y)$, written $f : (X, d_X, \mu_X) \to (Y, d_Y, \mu_Y)$, is a function $f : X \to Y$ satisfying the following conditions:
    \begin{itemize}
        \item $f : (X, d_X) \to (Y, d_Y)$ is $1$-Lipschitz
        \item $f_\# \mu_1 = \mu_2$ (i.e., $f$ is a Monge map).
    \end{itemize}
\end{definition}

Note that the $1$-Lipschitz condition implies that $f$ is continuous with respect to the topologies
induced by the two metrics making 
$f$ a measurable function between the two $\sigma$-algebras of Borel sets. In the case of finite mm-spaces, the topology induced on any finite mm-space will be the discrete topology. Therefore, any map between two finite mm-spaces will be continuous, and thus measurable, even if the $1$-Lipschitz condition is not satisfied. In~\cref{sec: dgw as edit}, we introduce another notion of morphism, called \emph{order-preserving mm-space function}, between finite mm-spaces where the $1$-Lipschitz condition is no longer required. However, we still require that these functions respect the metrics in an order-preserving sense.

\paragraph{Gromov-Wasserstein Distance.}
Let $(X,d_X, \mu_X)$ and $(Y,d_Y,\mu_Y)$ be two mm-spaces. A \emph{coupling} between $\mu_X$ and $\mu_Y$ is a probability measure $\mu$ on the product measurable space $X\times Y$ such that
\begin{itemize}
    \item ${\pi_X}_{\#} \mu = \mu_X$,
    \item ${\pi_Y}_{\#} \mu = \mu_Y$.
\end{itemize}

Let $\cpl(\mu_X, \mu_Y)$ denote the set of all couplings between $(X,\mu_X)$ and $(Y,\mu_Y)$

\begin{definition}[$p$-distortion]
    Let $(X,d_X,\mu_X)$ and $(Y,d_Y,\mu_Y)$ be two finite mm-spaces and let $\mu \in \cpl(\mu_X, \mu_Y)$ be a coupling. For $p\in [1,\infty)$, we define the \emph{$p$-distortion} of $\mu$ as
    \[
    \dis_p (\mu) := \left(\sum_{\substack{x,x' \in X \\ y,y' \in Y}} \bigl |d_X(x,x') - d_Y(y,y') \bigr |^p\mu(x,y)\mu(x',y')\right)^\frac{1}{p},
    \]
    and, for $p=\infty$, we define the \emph{$\infty$-distortion} of $\mu$ as
    \[
    \dis_\infty (\mu) = \max_{(x,y),(x',y') \in \supp(\mu)}|d_X(x,x') - d_Y(y,y')|.
    \]
\end{definition}

Any Monge map $f :(X,d_X,\mu_X) \to (Y,d_Y, \mu_Y)$ between two finite mm-spaces induces a coupling $\mu_f := \left(\mathrm{id}_X, f\right)_\# \mu_X \in \cpl(\mu_X, \mu_Y)$, where 
\begin{align*}
    (\mathrm{id}_X, f) : X &\to X\times Y \\
                        x &\mapsto (x,f(x)).
\end{align*}
This coupling $\mu_f$ has $p$-distortion equal to
\[
\dis_p (\mu_f) =  \left(\sum_{x,x'\in X} \bigl |d_X(x,x') - d_Y(f(x),f(x')) \bigr |^p\mu_X(x)\mu_X(x')\right)^\frac{1}{p} =: \dis_p(f),
\]
which we call the \emph{$p$-distortion of $f$} and denote by $\dis_p(f)$.
\begin{definition}[p-Gromov-Wasserstein distance~{\cite[Definition 5.7]{facundo-gw}}]
Let $(X,d_X,\mu_X)$ and $(Y,d_Y,\mu_Y)$ be two finite mm-spaces. For each $p \in [1,\infty]$, the \emph{p-Gromov-Wasserstein Distance} between the mm-spaces $X$ and $Y$ is defined as  
\[
\dgw{p}(X,Y) := \frac{1}{2}\inf_{\mu \in \cpl(\mu_X,\mu_Y)} \dis_p(\mu)
\]
\end{definition}

%%%%%%%%%%%%%%%%%%%%%%%%%%%%%%%%%%%%%%%%%%%%%%%%%%
%%%%%%%%%%%%%%%%%%%%%%%%%%%%%%%%%%%%%%%%%%%%%%%%%%
\section{\texorpdfstring{$\dgh$}{} and ~\texorpdfstring{$\dgw{p}$}{} as Edit Distances}
\label{sec:dgh is edit}
%%%%%%%%%%%%%%%%%%%%%%%%%%%%%%%%%%%%%%%%%%%%%%%%%%
%%%%%%%%%%%%%%%%%%%%%%%%%%%%%%%%%%%%%%%%%%%%%%%%%%

Edit distances quantify the dissimilarity between two objects by the minimal-cost sequence of edit operations needed to transform one into the other \cite{skiena2008algorithm}. For graphs \cite{sanfeliu1983distance}, such operations typically include the insertion, deletion, or relabeling of nodes and edges. This generalizes string edit distance to structured data and provides a flexible, intuitive framework for graph comparison. Variants of graph edit distance are widely used in pattern recognition, bioinformatics, and structural data analysis.

In this section, we define two categories: the category $\cmet$ of finite pseudo-metric spaces and the category $\cmm$ of finite metric measure (mm) spaces. For each category, we introduce a notion of cost for morphisms—standard cost in $\cmet$ and $p$-cost (for $p \in [1,\infty]$) in $\cmm$. These notions enable us to define an edit distance between finite pseudo-metric spaces and a $p$-edit distance between finite mm-spaces. Our main result is that the Gromov-Hausdorff distance arises as an instance of edit distance, and the $p$-Gromov-Wasserstein distance as an instance of $p$-edit distance.

\subsection{\texorpdfstring{$\dgh$}{} as an Edit Distance}

Let $X$ be a set and let $d_1$ and $d_2$ be two pseudo-metrics on $X$. We say that $d_2$ is a \emph{metric specialization} of $d_1$ if
\[
d_1(a,b) \geq d_1 (c,d) \implies d_2 (a,b) \geq d_2 (c,d)
\]
for every $a,b,c,d \in X$. In this case, we write $d_1\Longrightarrow d_2$.

\begin{definition}(Metric order-preserving functions)
    Let $(X,d_X)$ and $(Y,d_Y)$ be two pseudo-metric spaces. We say that a function $f: X \to Y$ is \emph{(metric) order-preserving} if 
\[
d_X(x_1, x_2) \geq d_X(x_3,x_4) \implies d_Y(f(x_1), f(x_2)) \geq d_Y(f(x_3), f(z_4)) 
\]
for every $x_1, x_2, x_3, x_4 \in X$.
\end{definition}

\begin{definition}[Pullback metric]
    Let $(X,d_X)$ be a (pseudo)-metric space. Let $Z$ be any set and $\phi : Z \rightarrow X$ be any function. The \emph{pullback pseudo-metric}, denoted $\phi^* d_X$, is defined as
    \begin{align*}
        \phi^*d_X := d_X \circ (\phi, \phi) : Z\times Z &\to \RR_+ \\
                                                (z_1, z_2) &\mapsto d_X (\phi(z_1), \phi(z_2)).
    \end{align*}
\end{definition}

Note that a function $f : Z \to X$ between two pseudo-metric spaces $(Z, d_Z)$ and $(X, d_X)$ is metric order-preserving if and only if $\phi^*d_X$ is a metric specialization of $d_Z$.

\paragraph{Category of Finite Pseudo-Metric Spaces.}

We now define the category of finite pseudo-metric spaces. Let $(X,d_X)$ and $(Y,d_Y)$ be two finite pseudo-metric spaces. We say that a function $f: X\to Y$ is a \emph{metric morphism} if $f$ is metric order-preserving and surjective. In this case, we write $(X,d_X) \xtwoheadrightarrow{f} (Y,d_Y)$. Observe that the composition of two metric morphisms is also a metric morphism. Let $\cmet$ denote the category whose objects are finite pseudo-metric spaces and whose morphisms are metric morphisms.

\begin{definition}[Cost of a metric morphism]
    Let $(X,d_X)$ and $(Y,d_Y)$ be two objects in $\cmet$ and assume that there is metric morphism $(X,d_X) \xtwoheadrightarrow{f} (Y,d_Y)$. We define the cost of this morphism as 
    \[
    \cost_\cmet (f) := \dis (f) = \max_{x_1, x_2 \in X} \left | d_X (x_1, x_2) - d_Y (f(x_1), f(x_2) ) \right |.
    \]
    
\end{definition} 

\begin{remark}\label{rem: morphisms in catmet and tripods}
    In our category of finite pseudo-metric spaces, morphisms are defined as specific surjections that echo the role of tripods in the realization of the Gromov–Hausdorff distance between bounded metric spaces given in \cref{thm: tripod version gh}. This formulation captures the essence of how tripods mediate between two spaces to realize the Gromov-Hausdorff distance between them.
\end{remark}

As every morphism in $\cmet$ has a nonnegative cost associated with it, we have the notion of edit distance, denoted $d_\cmet^E \left((X,d_X), (Y, d_Y)\right)$, between objects in $\cmet$, see~\cref{subsec: edit}. 

In the following proposition, we show that the edit distance between finite pseudo-metric spaces coincides with twice the Gromov-Hausdorff distance.

\begin{prop}\label{prop: gh as edit}
    For two finite psuedo-metric spaces $(X,d_X)$ and $(Y,d_Y)$ in $\cmet$, we have
    \[
    d_\cmet^E (X,Y) = 2 \dgh (X,Y).
    \]
\end{prop}

\begin{proof}
    Let 
    \begin{equation*}
    \begin{tikzcd}
	(X,d_X) \ar[r, leftrightarrow, "f_1"] & (Z_1,d_{Z_1}) \ar[r, leftrightarrow, "f_2"] & \cdots 
	\ar[r, leftrightarrow, "f_{n-1}"] & (Z_{n-1},d_{Z_{n-1}}) \ar[r, leftrightarrow, "f_n"] & (Y, d_Y)
    \end{tikzcd}        
    \end{equation*}
    be a path between $(X, d_X)$ and $(Y, d_Y)$ realizing $d_\cmet^{E} (X,Y)$. For any $i=1,\ldots,n$, the metric morphism $f_i$ is a surjection. Without loss of generality, for a fixed $i$, assume that $f_i : Z_i \twoheadrightarrow Z_{i+1}$. Then, $f_i$ it induces the following tripod
    \[
    f_i: Z_i \xtwoheadleftarrow{\mathrm{id}_{Z_i}} Z_i \xtwoheadrightarrow{f_i} Z_{i+1}
    \]    
    whose distortion is $\cost_\cmet (f_i)$. Hence, we have that $2 \dgh (Z_{i-1}, Z_{i}) \leq \cost_\cmet (f_i)$ for $i=1,...,n$, where $Z_0 = X$ and $Z_{n} = Y$. Thus,
    \begin{align*}
        2 \dgh (X,Y) &\leq \sum_{i=1}^n 2 \dgh (Z_{i-1}, Z_{i}) \\
        &\leq \sum_{i=1}^n \cost_\cmet (f_i) \\
        &= d_\cmet^{E} (X,Y).
    \end{align*}
    To see the other inequality, namely, $d_\cmet^{E} (X,Y)\leq 2\dgh(X,Y)$, let 
    \[
    \mathfrak{R}: X \xtwoheadleftarrow{\pi_X} Z \xtwoheadrightarrow{\pi_Y} Y
    \]
    be an optimal tripod realizing $\dgh (X,Y)$. That is, $2\dgh (X,Y) = \dis (\mathfrak{R})$. Let $d_1 := \pi_X^* d_X$ and $d_2 := \pi_Y^* d_Y$. Then, consider the family
    \[
    Z_t := (Z, d_t:=(1-t)d_1 + t d_2)
    \]
    for $t\in [0,1]$. Observe that, by construction, we have that $\dgh (Z_t, Z_s) \leq |t-s| \dgh (X,Y)$ for all $t,s\in[0,1]$.\footnote{This condition implies that the family $\{ Z_t\}_{t\in[0,1]}$ is a geodesic between $Z_0=X$ and $Z_1=Y$, see~\cite[Theorem 1.2]{Chowdhury2018}.} Let $z_1, z_2, z_3, z_4 \in Z$ be such that either $d_1(z_1, z_2) \neq d_1(z_3, z_4)$ or $d_2(z_1, z_2) \neq d_2(z_3, z_4)$. Then there exists at most one value $t \in [0,1]$ such that $d_t(z_1, z_2) = d_t(z_3, z_4)$. Let $0\leq t_1 < t_2 < ... <t_k \leq 1$ be the list of such $t_i$s arising from quadruples of points of $Z$. Let $c_0 < c_1 < ... < c_k$ be such that $t_i < c_i < t_{i+1}$ for $i=0,...,k$ where $t_0 = 0$ and $t_{k+1} = 1$. Then, the functions
    \begin{align*}
        g_{c_i , t_i} := \mathrm{id}_Z &: Z_{c_i} \to Z_{t_i} \\
        g_{c_i , t_{i+1}} := \mathrm{id}_Z &: Z_{c_i} \to Z_{t_{i+1}} 
    \end{align*}
       determines metric morphisms $(Z_{c_i}, d_{c_i}) \xtwoheadrightarrow{g_{c_i \to t_i}} (Z_{t_i}, d_{t_i})$ and $(Z_{c_i}, d_{c_i}) \xtwoheadrightarrow{g_{c_i \to t_{i+1}}} (Z_{t_{i+1}}, d_{t_{i+1}})$, respectively, with
    \begin{align*}
        \cost_\cmet (g_{c_i, t_i}) &= \dis (g_{c_i , t_i}) = 2 (c_i-t_i) \dgh(Z_{0}, Z_{1}) \\
        \cost_\cmet (g_{c_i,t_{i+1}}) &= \dis (g_{c_i , t_{i+1}}) = 2 (t_{i+1}-c_i) \dgh(Z_{0}, Z_{1})
    \end{align*}
    Thus,
    \begin{align*}
        d_\cmet^{E}(X,Y) &\leq \sum_{i=1}^k \cost_\cmet (g_{c_i , t_i}) + \cost_\cmet (g_{c_i , t_{i+1}}) \\
        &= \sum_{i=0}^k (t_{i+1} - t_i) 2 \dgh (X,Y) \\
        &= 2 \dgh (X,Y)
    \end{align*}
\end{proof}

\subsection{\texorpdfstring{$\dgw{p}$}{} as an Edit Distance}\label{sec: dgw as edit}

In this section, we define the category of finite mm-spaces and show that the $p$-Gromov-Wasserstein distance can be obtained as a $p$-edit distance. We begin by introducing the notion of \emph{order-preserving mm-space functions}, which will serve as morphisms in the category of finite mm-spaces.

\begin{definition}[Order-preserving mm-space function]\label{defn: order-preserving mm-space function}
    An \emph{order-preserving mm-space function} between two finite mm-spaces $(X, d_X, \mu_X)$ and $(Y, d_Y, \mu_Y)$, written $f : (X, d_X, \mu_X) \to (Y, d_Y, \mu_Y)$,
is a function $f : X \to Y$ satisfying the following conditions:
    \begin{enumerate}
        \item $f$ is metric order-preserving. That is, $$d_X(x_1, x_2) \leq d_X(x_3, x_4) \implies d_Y(f(x_1), f(x_2)) \leq d_Y (f(x_3), f(x_4))$$ for all $x_1, x_2, x_3, x_4 \in X$.\label{condition equiv to metric morphism}
        \item $f_\# \mu_1 = \mu_2$ (i.e., $f$ is a Monge map)\label{monge condition}.
    \end{enumerate}
\end{definition}

Observe that the composition of two order-preserving mm-space functions is also an order-preserving mm-space function. Note that, in the definition of order-preserving mm-space functions above, \cref{monge condition} guarantees that the function $f : X \to Y$ has to be surjective. Combined with~\cref{condition equiv to metric morphism}, this surjectivity implies that $(X,d_X) \xtwoheadrightarrow{f} (Y,d_Y)$ is a metric morphism.

Let $\cmm$ denote the category whose objects are finite mm-spaces and whose morphisms are order-preserving mm-space functions.

\begin{definition}[p-cost of an order-preserving mm-space function]
    Let $(X,d_X,\mu_X)$ and $(Y, d_Y, \mu_Y)$ be two finite mm-spaces and let $f : X\to Y$ be an order-preserving mm-space function. For $p \in [1,\infty)$, we define the p-cost of $f$ as
    \[
    \cost_\cmm^{p} (f) := \dis_p(f) = \left ( \sum_{x,x' \in X} \left| d_X(x,x') - d_Y (f(x), f(x')) \right|^p \mu_X (x) \mu_X (x') \right ) ^\frac{1}{p}.
    \]
    For $p=\infty$, we define the $\infty$-cost of $f$ as 
    \[
    \cost_\cmm^{\infty} (f) := \dis_\infty (f) =  \max_{x,x' \in X} \left| d_X(x, x') - d_Y (f(x), f(x')) \right|.
    \]
\end{definition}

For a fixed $p \in [1,\infty]$, we now have the notion of \emph{$p$-edit distance} as follows. Every morphism in $\cmm$ has a nonnegative $p$-cost associated with it. As discussed in \cref{subsec: edit}, we obtain an edit distance between objects of $\cmm$. We refer to this edit distance as the $p$-edit distance and denote it by $d_\cmm^{E,p}$.

In the following proposition, we show that the $p$-edit distance between finite mm-spaces coincides with twice the $p-$Gromov-Wasserstein distance.

\begin{prop}\label{prop: p-gw as edit}
    For every $p\in [1,\infty]$ and two finite mm-spaces spaces $(X,d_X,\mu_X)$ and $(Y,d_Y,\mu_Y)$ in $\cmm$ we have
    \[
    d_\cmm^{E,p} (X,Y) = 2 \dgw{p} (X,Y).
    \]    
\end{prop}

\begin{proof}
    Let 
    \begin{equation*}
    \begin{tikzcd}
	(X,d_X, \mu_X) \ar[r, leftrightarrow, "f_1"] & (Z_1,d_{Z_1}, \mu_{Z_1}) \ar[r, leftrightarrow, "f_2"] & \cdots 
	\ar[r, leftrightarrow, "f_{n-1}"] & (Z_{n-1},d_{Z_{n-1}}, \mu_{Z_{n-1}}) \ar[r, leftrightarrow, "f_n"] & (Y, d_Y, \mu_Y)
    \end{tikzcd}        
    \end{equation*}
    be a path between $(X, d_X, \mu_X)$ and $(Y, d_Y, \mu_Y)$ realizing $d_\cmm^{E,p} (X,Y)$. Since each order-preserving mm-space function $f_i$ is a Monge map, each $f_i$ induces a coupling $\mu_{f_i}$, as discussed in~\cref{subsec: mm-spaces}, whose p-distortion is $\dis_p(\mu_{f_i}) = \dis_p (f_i) = \cost_\cmm^p (f_i)$. Thus, we obtain that $2 \dgw{p} (Z_{i-1}, Z_{i}) \leq \cost_\cmm^p (f_i)$ for $i=1,...,n$, where $Z_0 = X$ and $Z_{n} = Y$. Thus,
    \begin{align*}
        2 \dgw{p} (X,Y) &\leq \sum_{i=1}^n 2 \dgw{p} (Z_{i-1}, Z_{i}) \\
        &\leq \sum_{i=1}^n \cost_\cmm^p (f_i) \\
        &= d_\cmm^{E,p} (X,Y).
    \end{align*}
    To see the other inequality, namely, $d_\cmm^{E,p} (X,Y)\leq 2\dgw{p}(X,Y)$, let $\mu$ be an optimal coupling between $\mu_X$ and $\mu_Y$. That is, $2\dgw{p} (X,Y) = \dis_p(\mu)$. Let $S = \supp(\mu) \subseteq X\times Y$ be the support of the measure $\mu$. Then, the family
    \[
    Z_t := (S, d_t:=(1-t)d_X + t d_Y, \mu)
    \]
    for $t\in [0,1]$ is a geodesic connecting $X$ and $Y$, i.e., $\dgw{p} (Z_t, Z_s) \leq |s-t|\dgw{p} (X,Y)$ for all $s,t\in[0,1]$, see~\cite[Theorem 3.1]{sturm-geodesic}. Let $(x_i,y_i) \in S$ for $i=1,2,3,4$ be such that either $d_X(x_1, x_2) \neq d_X(x_3, x_4)$ or $d_Y(y_1, y_2) \neq d_Y(y_3, y_4)$. Then there exists at most one value $t \in [0,1]$ such that $d_t((x_1,y_1), (x_2, y_2)) = d_t((x_3,y_3), (x_4, y_4))$. Let $0\leq t_1 < t_2 < ... <t_k \leq 1$ be the list of such $t_i$s arising from quadruples of points of $S$. Let $c_0 < c_1 < ... < c_k$ be such that $t_i < c_i < t_{i+1}$ for $i=0,...,k$ where $t_0 = 0$ and $t_{k+1} = 1$. Then, the maps $f_{c_i \to t_i} : Z_{c_i} \to Z_{t_i}$ and $f_{c_i \to t_{i+1}} : Z_{c_i} \to Z_{t_{i+1}}$, given by the identity map on $S$, are order-preserving mm-space functions with $\cost_\cmm^p (c_i\to t_i) = \dis_p (f_{c_i \to t_i}) = 2 (c_i-t_i) \dgw{p}(Z_{0}, Z_{1
    })$ as the family of $Z_t$s is a geodesic. Thus,
    \begin{align*}
        d_\cmm^{E,p}(X,Y) &\leq \sum_{i=1}^k \cost_\cmm^p(f_{c_i \to t_i}) + \cost_\cmm^p(f_{c_i \to t_{i+1}}) \\
        &= \sum_{i=0}^k (t_{i+1} - t_i) 2 \dgw{p} (X,Y) \\
        &= 2 \dgw{p} (X,Y)
    \end{align*}
\end{proof}

%%%%%%%%%%%%%%%%%%%%%%%%%%%%%%%%%%%%%%%%%%%%%%%%%%
%%%%%%%%%%%%%%%%%%%%%%%%%%%%%%%%%%%%%%%%%%%%%%%%%%
\section{A Functorial Construction of (Weighted) Vietoris-Rips Filtration}
\label{sec:vr_functor}
%%%%%%%%%%%%%%%%%%%%%%%%%%%%%%%%%%%%%%%%%%%%%%%%%%
%%%%%%%%%%%%%%%%%%%%%%%%%%%%%%%%%%%%%%%%%%%%%%%%%%

In this section, we introduce the category of filtrations, $\Fil$, and the category of weighted filtrations, $\wfil$. We show that the Vietoris-Rips filtration is a functor from $\cmet$ to $\Fil$. Moreover, we introduce the notion of weighted Vietoris-Rips filtration and we show that it is a functor from $\cmm$ to $\wfil$.

\subsection{Unweighted Functorial Vietoris-Rips}

\paragraph{Category of Unweighted Filtrations.}
Let $V$ be a finite set and let $K_V$ be a simplicial complex whose vertex set is $V$. Let $\Ffunc : Q \to \subcx(K_V)$ be a $1$-parameter filtration of the finite simplicial complex $K_V$. That is, $\Ffunc$ is an order-preserving map $\Ffunc : Q \to \subcx(K_V)$ where $Q\subseteq [0,\infty)$ is finite, $0\in Q$ and $\Ffunc (\top_Q) = K_V$. Let $Z$ be any finite set and $\pi : Z \twoheadrightarrow V$ be a surjection. For any $q\in Q$, let $\Ffunc_Z^{\pi} \subseteq 2^Z \setminus \emptyset $ denote the smallest simplicial complex containing $\left\{ \pi^{-1} (\sigma) \mid \sigma \in \Ffunc(q) \right\}$. Let $K_Z$ denote the simplicial complex $\Ffunc(\top_Q)$ where $\top_Q$ is the maximum element of $Q$. Then, 
\[
\Ffunc_Z^{\pi} : Q \to \subcx  (K_Z)
\]
is a $1$-parameter filtration of the finite simplicial complex $K_Z$, called the \emph{pullback filtration} of $\Ffunc$ along $\pi$.

\begin{definition}[Filtration-preserving morphisms]\label{defn: filtration-presv morphism}
    Let $V$ and $W$ be two finite sets and $K_V$ and $K_W$ be two simplicial complexes over the vertex sets $V$ and $W$ respectively. Let $\Ffunc : Q  \to \subcx (K_V)$ and $\Gfunc : R \to \subcx(K_W)$ be two $1$-parameter filtrations. A \emph{filtration-preserving morphism} from $\Ffunc$ to $\Gfunc$ is a tuple $(f,g, Z, \pi_V, \pi_W)$ such that
    \begin{itemize}
        \item $Z$ is a finite set and $V \xtwoheadleftarrow{\pi_V} Z \xtwoheadrightarrow{\pi_W} W$ is a tripod,
        \item $f: Q \leftrightarrows R : g$ is a Galois connection,
        \item $\Gfunc_Z^{\pi_W} (r) = \Ffunc_Z^{\pi_V} (g(r))$ for all $r\in R$.
    \end{itemize}
\end{definition}
Let $\Fil$ denote the category whose objects are $1$-parameter filtrations of finite simplicial complexes and whose morphisms are filtration-preserving morphisms.

\begin{remark}
    Organizing the collection of $1$-parameter filtrations into the category $\Fil$ through the filtration-preserving morphisms serves as a middle step in constructing the functorial pipeline that goes from finite (pseudo)-metric spaces to persistence diagrams. Note that functoriality is the main tool that we utilize in order to prove stability with respect to the edit distance. In~\cref{prop: unweighted vr is functor}, we show that assigning the Vietoris-Rips filtration to a finite (pseudo)-metric space is a functor from $\cmet $ to $ \Fil$. Later, in~\cref{prop: functoriality of pd}, we see that assigning the degree-$d$ persistence diagram to a filtration is a functor from $\Fil$ to $\Dgm$, the category of persistence diagrams. Moreover, all these functors respect the cost of morphisms. Thus, we obtain the edit distance stability.
\end{remark}

\begin{definition}[Cost of a filtration-preserving morphism]
    Let $\Ffunc : Q  \to \subcx (K_V)$ and $\Gfunc : R \to \subcx(K_W)$ be two $1$-parameter filtrations in $\Fil$ and let $(f,g, Z, \pi_V, \pi_W)$ be a filtration-preserving morphism from $\Ffunc$ to $\Gfunc$. We define the cost of this morphism as
    \[
    \cost_\Fil ((f,g, Z, \pi_V, \pi_W)) := \dis (f) = \max_{q_1, q_2 \in Q} \bigl| |q_1 - q_2| - |f(q_1) - f(q_2)| \bigr|.
    \]
\end{definition}

As every morphism in $\Fil$ has a nonnegative cost associated with it, we have the notion of edit distance, denoted $d_\Fil^E \left(\Ffunc, \Gfunc\right)$, between objects in $\Fil$, see~\cref{subsec: edit}.

\paragraph{Unweighted Vietoris-Rips.}
Let $(X,d_X)$ be a finite (pseudo)-metric space. The Vietoris-Rips complex of $X$ at scale $r \in [0,\infty)$, denoted $\VR((X, d_X), r)$, is defined to be the abstract simplicial complex constructed by forming a simplex for every subset of $X$ that has diameter at most $r$. As $r$ varies, this construction determines a filtration 
called the Vietoris-Rips filtration of $X$. 

Let
\begin{align*}
    D(X,d_X) &:= \{ r\in [0,\infty) \mid \exists x_1, x_2 \text{ such that } r = d_X (x_1, x_2)\} \\
                & = \ima (d_X : X\times X \to \RR_+).
\end{align*}

\begin{definition}\label{def:VR-filt}[Vietoris-Rips filtration]
    Let $(X,d_X)$ be a finite (pseudo)-metric space. The \emph{Vietoris-Rips filtration} of $X$ is defined to be the $1$-parameter filtration given by
    \begin{align*}
            \VR(X,d_X) : D(X, d_X) &\to \subcx (2^X \setminus \emptyset) \\
                        r &\mapsto \VR((X,d_X),r)
    \end{align*}
\end{definition}

\begin{prop}\label{prop: unweighted vr is functor}
    The assignment 
    \[
        (X,d_X) \to \left( \VR(X,d_X) : D(X,d_X) \to \subcx(2^X \setminus \emptyset) \right)
    \]
    is a functor from $\cmet$ to $\Fil$. Moreover, for any $(X,d_X)$ and $(Y,d_Y)$ in $\cmet$, we have that
    \[
    d_\Fil^E (\VR(X,d_X), \VR(Y,d_Y)) \leq 2 \cdot d_\cmet^E ((X,d_X),(Y,d_Y)) = 4 \cdot \dgh((X,d_X),(Y,d_Y)).
    \]
\end{prop}

To prove~\cref{prop: unweighted vr is functor}, we will need the following lemma.

\begin{lemma}\label{lem: technical lemma for functorial VR}
    Assume that $Z$ is a finite set and $d_1$ and $d_2$ are two (pseudo) metrics on $Z$ such that $d_2$ is a metric specialization of $d_1$. Let $f : D(Z,d_1) \to D(Z,d_2)$ be given by $f(d_1(z_1, z_2)) = d_2(z_1, z_2)$. Then,
    \begin{enumerate}
        \item $f$ is a well-defined order-preserving map and admits a right adjoint $g : D(Z,d_2) \to D(Z, d_1)$,
        \item For every $r\in D(Z,d_2)$, if $d_1 (z,z') = g(r)$, then $d_2(z,z') = r$, \label{technical condition required for filtraton-preserving}
        \item $\VR((Z,d_2), r) = \VR((Z,d_1), g(r))$ for every $r \in D(Z,d_2)$. \label{classical filtration-preserving condition}
    \end{enumerate}
\end{lemma}

\begin{proof}
  $ $
    \begin{enumerate}
        \item Assume that $d_1(z_1, z_2) = d_1(z_3, z_4)$, then we have $d_1(z_1, z_2) \leq d_1(z_3, z_4)$ and $d_1(z_1, z_2) \geq d_1(z_3, z_4)$. As $d_2$ is a metric specialization of $d_1$, we obtain $d_2(z_1, z_2) \leq d_2(z_3, z_4)$ and $d_2(z_1, z_2) \geq d_2(z_3, z_4)$. Hence, $d_2(z_1, z_2) = d_2(z_3, z_4)$, and $f$ is well-defined. The order-preserving property of $f$ follows from the fact that $d_2$ is a metric specialization of $d_1$. The right adjoint of $f$ is given by
        \begin{align*}
            g : D(Z, d_2) &\to D(Z, d_1) \\
                   r &\mapsto \max f^{-1} ([0,r]) = \max_{\substack{z,z' \in Z \\ d_2 (z, z')\leq r}} d_1 (z,z'). 
        \end{align*}
        \item Assume that $d_1 (z,z') = g(r)$.
        Let $a,b \in Z$ be such that $d_2 (a,b) = r$. Then, $d_1(a,b) \leq g(r) = d_1 (z,z')$ by definition of $g$. Therefore, $r = d_2(a,b) \leq d_2 (z,z')$ since $d_2$ is a metric specialization of $d_1$. By definition of $g$, there exists $c,d \in Z$ such that $d_1 (c,d) = g(r)$ and $d_2 (c,d)\leq r$. Therefore, $d_1 (z,z') =g(r) = d_1 (c,d)$. In particular, we have that $d_1 (z,z') \leq d_1 (c,d)$, which implies that $d_2 (z,z') \leq d_2 (c,d) \leq r$ as $d_2$ is a metric specialization of $d_1$. Hence, $d_2(z,z') = r$. 
        \item It suffices to verify that $\VR((Z,d_2), r)$ and $\VR((Z,d_1), g(r))$ have the same set of edges for every $r \in D(Z,d_2)$. Let $\{a,b \}$ be an edge in $\VR((Z,d_1), g(r))$. That is, $d_1 (a,b) \leq g(r)$. Let $z,z' \in Z$ be such that $d_1(z,z') = g(r)$. Then, $d_1(a,b)\leq g(r) = d_1(z,z')$. Since $d_2$ is a metric specialization of $d_1$, we obtain $d_2(a,b)\leq d_2(z,z')$. Moreover, by~\cref{technical condition required for filtraton-preserving}, we have that $d_2(z,z') = r$. Thus, $d_2(a,b) \leq d_2(z,z') = r$. Therefore, $\{a, b \}$ is an edge in $\VR((Z,d_2), r)$. So, we obtain $\VR((Z,d_2), r) \supseteq \VR((Z,d_1), g(r))$.

        Now, assume that $d_2(c,d) \leq r$ for some $c,d \in Z$. Then, $d_1(c,d) \leq g(r)$ by definition of $g$. Thus, $\VR((Z,d_2), r) \subseteq \VR((Z,d_1), g(r))$.

    \end{enumerate}
\end{proof}

\begin{proof}[Proof of~\cref{prop: unweighted vr is functor}]
    Let $(X,d_X)$ and $(Y,d_Y)$ be two finite (pseudo) metric spaces in $\cmet$ and let $(X, d_X) \xtwoheadrightarrow{h} (Y,d_Y)$ be a metric morphism. Let
    \[
    \mathfrak{R}: X \xtwoheadleftarrow{\phi_X} Z \xtwoheadrightarrow{\phi_Y} Y
    \]
    be the tripod determined by $h$. That is, $Z = X$, $\phi_X = \mathrm{id}_X$, and $\phi_Y = h$. Note that, as mentioned earlier, $d_2 := \phi_Y^* d_Y$ is a metric specialization of $d_1 := \phi_X^* d_X =\mathrm{id}_X^* d_X = d_X $. Observe that the pullback filtration of Vietoris-Rips filtrations is the Vietoris-Rips filtration of the pullback metric. That is, we have
    \begin{itemize}
        \item $\VR(X,d_X)_Z^{\phi_X} = \VR(Z,\phi_X^* d_X)= \VR(Z,d_1)$,
        \item $\VR(Y,d_Y)_Z^{\phi_Y} = \VR(Z,\phi_Y^* d_Y)= \VR(Z,d_2)$.
    \end{itemize}
    Observe that $D(X,d_X) = D(Z,d_1)$ and $D(Y,d_Y) = D(Z,d_2)$. Thus, by~\cref{lem: technical lemma for functorial VR}, we obtain a Galois connection 
    \begin{align*}
        f : D(X,d_X) = D(Z,d_1) \leftrightarrows D(Y,d_Y) = D(Z,d_2) : g,
    \end{align*}
    where $f$ if given by $f(d_1(z,z')) = d_2(z,z')$. By~\cref{condition equiv to metric morphism} in~\cref{lem: technical lemma for functorial VR}, we obtain
    \[
    \VR(Y,d_Y)_Z^{\phi_Y} (r) = \VR((Z,d_2),r) = \VR((Z,d_1), g(r)) = \VR(X,d_X)_Z^{\phi_X}.
    \]
    Hence, the tuple $(f,g,Z,\phi_X,\phi_Y)$ is a filtration-preserving morphism from $\VR(X,d_X)$ to $\VR(Y,d_Y)$.

    In order to show that $d_\Fil^E (\VR(X,d_X), \VR(Y,d_Y)) \leq 2 \cdot d_\cmet^E ((X,d_X),(Y,d_Y))$, by~\cref{prop: general edit distance stability}, it suffices to verify that 
    \[
    \cost_\Fil \left((f,g,Z,\phi_X,\phi_Y) \right) \leq 2 \cdot \cost_\cmet \left ( (X,d_X) \xrightarrow{h} (Y,d_Y) \right).
    \]
    We obtain this inequality as follows.
    \begin{align*}
        \cost_\Fil \left((f,g,Z,\phi_X,\phi_Y) \right) &= \dis(f) & \\
                                                   &= \max_{r,r' \in D(X,d_X)} \bigl | |r-r'| - |f(r) - f(r')|  \bigr| & \\
                                                   &= \max_{r,r' \in D(X,d_X)} \bigl | (r-r') - (f(r) - f(r'))  \bigr| &\text{Since $f$ is order-preserving}\\
                                                   &\leq 2 \cdot \max_{r \in D(X,d_X)} |r-f(r)| &\text{by the triangle inequality}\\
                                                   &= 2 \cdot \max_{z,z' \in Z} \bigl | d_1(z,z') - d_2(f(z), f(z'))  \bigr| & \\
                                                   &= 2 \cdot \max_{z,z' \in Z} \bigl | d_X(\phi_X(z), \phi_X(z')) - d_Y(\phi_Y(z), \phi_Y(z'))   \bigr| & \\
                                                   &= 2 \cdot \max_{x,x' \in X} \bigl | d_X(x, x') - d_Y(h(x), h(x'))   \bigr| & \\
                                                   &= 2 \cdot \dis(h) = 2 \cdot \cost_\cmet (h). &
    \end{align*}
\end{proof}

\subsection{Weighted Functorial Vietoris-Rips}
Let $Q \subseteq [0,\infty)$ be a finite subset containing 0. Let $\mu_Q$ be a probability measure with full support on $Q$ and let $d_Q$ be the restriction of the Euclidean distance on $[0,\infty)$ to $Q$. Then, $(Q, d_Q, \mu_Q)$ is a finite mm-space. Moreover, $Q$ inherits the linear order on $[0,\infty)$. 
\begin{definition}
    A quadruple $(Q, d_Q, \mu_Q, \leq_Q)$ is called a finite linear mm-space if
    \begin{itemize}
        \item $Q \subseteq [0,\infty)$ is finite and $0\in Q$,
        \item $d_Q$ is the restriction of the Euclidean distance on $[0,\infty)$ to $Q$,
        \item $\mu_Q$ is a probability measure with full support on $Q$,
        \item $\leq_Q$ is the restriction of the linear order on $[0,\infty)$ to $Q$.
    \end{itemize}
\end{definition}

\paragraph{Category of Weighted Filtrations.}
We now introduce the category of weighted filtrations.
\begin{definition}[Weighted filtration]\label{defn: weighted filtration}
    A $1$-parameter filtration $\Ffunc : Q \to \subcx(K)$ is called a \emph{weighted filtration} if there is a probability measure $\mu_Q$ on $Q$ with full support. We refer to $\mu_Q$ as the \emph{weight} on the filtration. We denote the weighted filtration by $w\textnormal{-}\Ffunc : (Q,\mu_Q) \to \subcx(K)$.
\end{definition}

\begin{definition}[Weighted filtration-preserving morphisms]
    Let $V$ and $W$ be two finite sets and $K_V$ and $K_W$ be two simplicial complexes over the vertex sets $V$ and $W$ respectively. Let $w\textnormal{-}\Ffunc : Q  \to \subcx (K_V)$ and $w\textnormal{-}\Gfunc : R \to \subcx(K_W)$ be two weighted filtrations. That is, $\Ffunc$ and $\Gfunc$ are $1$-parameter filtrations and there are probability measure $\mu_Q$ and $\mu_R$ with support on $Q$ and $R$ respectively. A \emph{weighted filtration-preserving morphism} from $\Ffunc$ to $\Gfunc$ is a tuple $(f,g, Z, \pi_V, \pi_W)$ such that
    \begin{itemize}
        \item $Z$ is a finite set and $V \xtwoheadleftarrow{\pi_V} Z \xtwoheadrightarrow{\pi_W} W$ is a tripod,
        \item $f: Q \leftrightarrows R : g$ is a Galois connection,
        \item $\Gfunc_Z^{\pi_W} (r) = \Ffunc_Z^{\pi_V} (g(r))$ for all $r\in R$,
        \item $f_\# \mu_Q = \mu_R$.
    \end{itemize}
\end{definition}

Note that a weighted filtration-preserving morphism is a filtration-preserving morphism with the extra condition that the left adjoint in the filtration-preserving morphism is also a Monge map. Let $\wfil$ denote the category whose objects are weighted filtrations and whose morphisms are weighted filtration-preserving morphisms.

\begin{definition}[$p$-cost of a weighted filtration-preserving morphism]
    Let $\Ffunc : Q \to \subcx(K_V)$ and $\Gfunc : R \to \subcx(K_W)$ be two weighted filtrations with weights $\mu_Q$ and $\mu_R$ respectively. Let $(f,g,Z,\pi_V,\pi_W)$ be a weighted filtration-preserving morphism from $\Ffunc$ to $\Gfunc$. For $p\in [1,\infty)$, we define the $p$-cost of this morphism as
    \[
    \cost_{\wfil}^p ((f,g,Z,\pi_V,\pi_W)) := \dis_p(f) = \left ( \sum_{q_1, q_2 \in Q} \bigl |  |q_1 - q_2| - |f(q_1)-f(q_2)|  \bigr |^p \mu_Q (q_1)\; \mu_Q(q_2) \right ) ^{\frac{1}{p}}.
    \]
    For $p=\infty$, we define the $\infty$-cost of this morphism as
    \[
    \cost_{\wfil}^\infty ((f,g,Z,\pi_V,\pi_W)) := \dis_\infty (f) = \sup_{q_1,q_2 \in Q} \bigl | |q_1 - q_2| - |f(q_1)-f(q_2)| \bigr |.
    \]
\end{definition}

For a fixed $p \in [1,\infty]$, we now have the notion of \emph{$p$-edit distance} as follows. Every morphism in $\wfil$ has a nonnegative $p$-cost associated with it. As discussed in \cref{subsec: edit}, we obtain an edit distance between objects of $\wfil$. We refer to this edit distance as the $p$-edit distance and denote it by $d_\wfil^{E,p}$.

\paragraph{Weighted Vietoris-Rips.}
Let $(X,d_X, \mu_X)$ be a finite mm-space. Recall that $$\VR(X,d_X) : D(X,d_X) \to \subcx (2^X \setminus \emptyset)$$ is a $1$-parameter filtration. We now assign a probability measure to $D(X,d_X)$ and obtain a weighted filtration as follows.

\begin{definition}[Global distribution of distances, {\cite[Definition 5.4]{facundo-gw}}]\label{defn: global dist of dist}
    Let $(X,d_X,\mu_X)$ be a finite mm-space. Consider the surjection
    \[
    d_X : X \times X \to D(X,d_X) = \ima(d_X).
    \]
    We define the \emph{global distribution of distances} of $X$ as the pushforward measure
    \[
    \mu_{GDD(X)} := (d_X)_\# (\mu_X \otimes \mu_X).
    \]
\end{definition}

\begin{definition}\label{def:wVR-filt}
    The weighted Vietoris-Rips filtration of a finite mm-space $(X,d_X,\mu_X)$ is the Vietoris-Rips filtration of the metric space $(X,d_X)$, $\VR(X,d_X)$, with weights $\mu_{GDD(X)}$. In order to distinguish the weighted Vietoris-Rips filtration from the unweighted one, we use the notation $\WVR (X,d_X)$ to indicate the presence of weights.
\end{definition}

\begin{prop}\label{prop: weighted vr is functor}
    The assignment 
    \[
        (X,d_X,\mu_X) \to \left( \WVR(X,d_X) : D(X,d_X) \to \subcx(2^X \setminus \emptyset) \right)
    \]
    is a functor from $\cmm$ to $\wfil$. Moreover, for any $(X,d_X,\mu_X)$ and $(Y,d_Y,\mu_Y)$ in $\cmm$, we have that
    \begin{align*}
    d_\wfil^{E,p} (\WVR(X,d_X,\mu_X), \WVR(Y,d_Y,\mu_Y)) &\leq 2 \cdot d_\cmm^{E,p} ((X,d_X,\mu_X),(Y,d_Y,\mu_Y)) \\ &= 4 \cdot \dgw{p}((X,d_X,\mu_X),(Y,d_Y,\mu_Y)).
    \end{align*}
\end{prop}

\begin{proof}
    Let $h : (X,d_X,\mu_X) \to (Y,d_Y,\mu_Y)$ be a order-preserving mm-space function. We now show that $h$ induces a weighted filtration-preserving morphism from $\WVR(X,d_X,\mu_X)$ to $\WVR (Y,d_Y, \mu_Y)$. Since $h$ is a Monge map, it must be surjective. Let $gr(h) := \{ (x,h(x)) \in X\times Y \mid x\in X \}$ be the graph of $h$. Then, we obtain the following tripod
    \[
    \mathfrak{R}_h : X \xtwoheadleftarrow{\pi_X} gr(h) \xtwoheadrightarrow{\pi_Y} Y,
    \]
    where $\pi_X$ and $\pi_Y$ are the component-wise projections from $gr(h)$ to $X$ and $Y$ respectively. Let $d_1 = \pi_X^* d_X$ and $d_2 = \pi_Y^* d_Y$. Then, we have that $d_2$ is a metric specialization of $d_1$. Observe that the pullback filtration of Vietoris-Rips filtration is the Vietoris-Rips filtration of the pullback metric. That is, 
    \begin{itemize}
        \item $\VR(X,d_X)_{gr(h)}^{\phi_X} = \VR(gr(h),\phi_X^* d_X)= \VR(gr(h),d_1)$,
        \item $\VR(Y,d_Y)_{gr(h)}^{\phi_Y} = \VR(gr(h),\phi_Y^* d_Y)= \VR(gr(h),d_2)$.
    \end{itemize}
    Then, by~\cref{lem: technical lemma for functorial VR}, we have that the function $f : D(X,d_X) = D(gr(h), d_1) \to D(gr(h), d_2) = D(Y,d_Y)$ defined by $f (d_X(x,x')) = d_Y (h(x), h(x'))$ has a right adjoint $g$ and the tuple $(f,g,gr(h),\pi_X,\pi_Y)$ is a filtration-preserving morphism. To conclude that we obtain a weighted filtration-preserving morphism, we need to check that $f$ is a Monge map. Recall that the weights in the weighted Vietoris-Rips filtration are given by the global distribution of distances. So, we need to check that $f_\# (\mu_{GDD(X)}) = \mu_{GDD(Y)}$. This equality follows from the commutative diagram
    \begin{equation*}
        \begin{tikzcd}
        X\times X \arrow[r, "d_X"] \arrow[d, "{(h,h)}"'] & {D(X,d_X)} \arrow[d, "f"] \\
        Y\times Y \arrow[r, "d_Y"]                      & {D(Y,d_Y)}               
        \end{tikzcd}
    \end{equation*}
    and the fact that $(h,h)_\# (\mu_X \otimes \mu_X) = \mu_Y \otimes \mu_Y$. Now, by~\cref{prop: general edit distance stability}, it suffices to show that $\cost_{\wfil}^{p} ((f,g,gr(h), \pi_X,\pi_Y) \leq 2 \cdot \cost_{\cmm}^{p} (h)$ in order to show that
    \[
    d_\wfil^{E,p} (\WVR(X,d_X,\mu_X), \WVR(Y,d_Y,\mu_Y)) \leq 2 \cdot d_\cmm^{E,p} ((X,d_X,\mu_X),(Y,d_Y,\mu_Y)).
    \]

    We obtain this inequality, when $p$ is finite, as follows.
    \footnotesize
    \begin{align*}
        &\left( \cost_{\wfil}^{p} ((f,g,gr(h), \pi_X,\pi_Y)) \right)^p \\
         &= \left( \dis_p(f) \right)^p \\ 
        &= \sum_{r, r' \in D(X,d_X)} \bigl | |r-r'| - |f(r) - f(r')|  \bigr |^p  \mu_{GDD(X)}(r) \mu_{GDD(X)}(r') \\
        &= \sum_{\substack{x_1, x_2,\\ x_3, x_4 \in X }} \Big | |d_X(x_1, x_2)-d_X(x_3, x_4)| - |d_Y(h(x_1), h(x_2)) - d_Y(h(x_3), h(x_4))| \Big |^p \mu(x_1)\mu(x_2)\mu(x_3)\mu(x_4) \\
        &\leq \sum_{\substack{x_1, x_2,\\ x_3, x_4 \in X}} \Big | d_X(x_1, x_2) - d_Y (h(x_1), h(x_2)) - (d_X(x_3, x_4) - d_Y (h(x_3), h(x_4))) \Big |^p \mu(x_1)\mu(x_2)\mu(x_3)\mu(x_4) \\
        &\leq \sum_{\substack{x_1, x_2,\\ x_3, x_4 \in X}} 2^{p-1} \Big ( |d_X(x_1, x_2) - d_Y (h(x_1), h(x_2))|^p + |d_X(x_3, x_4) - d_Y (h(x_3), h(x_4))|^p \Big ) \mu(x_1)\mu(x_2)\mu(x_3)\mu(x_4) \\
        &= 2^p (\dis_p(h))^p \\
        &= 2^p \cdot \left( \cost_{\cmm}^{p} (h) \right)^p.
    \end{align*}
    \normalsize
    When $p = \infty$, we obtain this inequality as follows.
    \begin{align*}
         \cost_{\wfil}^{\infty} ((f,g,gr(h), \pi_X,\pi_Y))   &=  \dis_\infty(f) \\
         &= \max_{r, r' \in D(X,d_X)} | |r-r'| - |f(r) - f(r')| \\
         &= \max_{\substack{x_1, x_2,\\ x_3, x_4 \in X }} \Big | |d_X(x_1, x_2)-d_X(x_3, x_4)| - |d_Y(h(x_1), h(x_2)) - d_Y(h(x_3), h(x_4))| \Big | \\
         &\leq \max_{\substack{x_1, x_2,\\ x_3, x_4 \in X}} \Big | d_X(x_1, x_2) - d_Y (h(x_1), h(x_2)) - (d_X(x_3, x_4) - d_Y (h(x_3), h(x_4))) \Big | \\
         &\leq  \max_{\substack{x_1, x_2,\\ x_3, x_4 \in X}}  \Big ( |d_X(x_1, x_2) - d_Y (h(x_1), h(x_2))| + |d_X(x_3, x_4) - d_Y (h(x_3), h(x_4))| \Big ) \\
         &= 2 \cdot \cost_{\cmm}^{\infty} (h)
    \end{align*}
\end{proof}

%%%%%%%%%%%%%%%%%%%%%%%%%%%%%%%%%%%%%%%%%%%%%%%%%%
%%%%%%%%%%%%%%%%%%%%%%%%%%%%%%%%%%%%%%%%%%%%%%%%%%
\section{Functorial (Weighted) Persistence Diagrams}
\label{sec:ph_functor}
%%%%%%%%%%%%%%%%%%%%%%%%%%%%%%%%%%%%%%%%%%%%%%%%%%
%%%%%%%%%%%%%%%%%%%%%%%%%%%%%%%%%%%%%%%%%%%%%%%%%%

In this section, we introduce the category of persistent diagrams, $\Dgm$, and the category of weighted persistence diagrams. $\wDgm$. We show that the degree-$d$ (weighted) persistence diagram assignment to a (weighted) filtration is a functor. As a result of this functoriality, we obtain a stability result.

%%%%%%%%%%%%%%%%%%%%%%%%%%%%%%%%%%%%%%%%%%%%%%%%%%
\subsection{(Unweighted) Persistence Diagrams}\label{subsec: unweighted persistence diagrams}
%%%%%%%%%%%%%%%%%%%%%%%%%%%%%%%%%%%%%%%%%%%%%%%%%%

Persistence diagrams are non-negative integer-valued functions on the set of intervals of a finite linear poset $Q$. The content presented in this section largely revisits the outcomes and framework established in~\cite{mccleary2022edit}. We now define the set of intervals of a poset.

\begin{definition}
    Let $Q$ be any poset. For $q_1 \leq q_2 \in Q$, we define the \emph{interval} $[q_1, q_2]$ to be the set $[q_1, q_2] := \{ q \in Q \mid q_1 \leq q \leq q_2 \}$. Then, we define the set of all intervals as 
    \[
    \bar Q := \{ [q_1, q_2] \mid q_1 \leq q_2 \in Q \}.
    \]
    We denote by $\diag (\bar Q)$ the \emph{diagonal} of $\bar Q$. That is,
    \[
    \diag (\bar Q) := \{ [q,q] \mid q\in Q \}
    \]
    Let $\leq_\times$ denote the product order on $Q\times Q$. Restricting $\leq_\times$ to $\bar Q$, the set of intervals is endowed with a partial order. Namely, for $[q_1, q_2], [q_3, q_4] \in \bar Q$, we write $[q_1, q_2] \leq_\times [q_3, q_4]$ if $q_1 \leq q_3$ and $q_2\leq q_4$.

\end{definition}

\begin{definition}
    Let $Q$ be a poset and $d_Q$ be a metric on $Q$. We define the metric $d_{\bar Q}$ on $\bar Q$ as 
    \[
    d_{\bar Q} ([q_1, q_2], [q_3, q_4]) := \max (d_Q(q_1, q_3), d_Q(q_2, q_4)).
    \]
    For a subposet $Q \subseteq [0,\infty)$, and for two intervals $I= [q_1, q_2]$ and $J = [q_3, q_4]$ in $\bar Q$, we define
    \[
    I-J := (q_1 - q_3, q_2 - q_4) \in \RR^2.
    \]
    With this notation, we have 
    \[
    d_{\bar Q} (I,J) = ||I-J||_\infty,
    \]
    where $||\cdot||_\infty$ denotes the $\ell$-infinity norm.
\end{definition}

\begin{remark}
    A Galois connection $f : Q \leftrightarrows R : g$ induces a Galois connection on the poset of intervals
    \[ \bar f : \bar Q \leftrightarrows \bar R\ : \bar g, \]
    where $\bar f$ and $\bar g$ are defined component-wisely.
\end{remark}

\begin{definition}
    Let $Q\subseteq [0,\infty)$ be a finite subset containing $0$. A \emph{persistence diagram} is a function (not necessarily order-preserving) $\sigma : \bar{Q} \to \NN$.
\end{definition}

\begin{definition}[Persistence diagram morphism, {\cite[Definition 6.1]{mccleary2022edit}}]\label{defn: persistence dgm morphism}
    A \emph{persistence diagram morphism} from a $\sigma : Q \to \NN$ to $\tau : R \to \NN$ is a Galois connection $f : P \leftrightarrows  Q : g$ such that $\bar{f}_\# \sigma = \tau$ on $\bar Q\setminus \diag (\bar Q)$. That is, 
    \[
    \tau (J) = \sum_{\substack{I \in \bar Q \\ \bar f (I) = J}} \sigma (I)
    \]
    for all $J \in \bar Q \setminus \diag(\bar Q)$.
\end{definition}

Let $\Dgm$ denote the category whose objects are persistence diagrams and whose morphisms are persistence diagram morphisms.

\begin{definition}[Cost of a persistence diagram morphism]
    Let $ f : Q \leftrightarrows R : g$ be a morphism from a persistence diagram $\sigma : \bar Q \to \NN$ to $\tau : \bar R \to \NN$. We define the cost of this morphism as
    \begin{align*}
    \cost_{\Dgm} (f : Q \leftrightarrows R : g) :&= \dis (\bar f) \\
                                            &= \max_{I,J \in \bar Q} \big | ||I-J||_\infty - ||\bar f (I) - \bar f (J)||_\infty \big |.
    \end{align*}
\end{definition}

As every morphism in $\Dgm$ has a nonnegative cost associated with it, we have the notion of edit distance, denoted $d_\Dgm^E$, between objects in $\Dgm$, see~\cref{subsec: edit}.

\begin{prop}[{\cite[Proposition 3.4]{mccleary2022edit}}]\label{prop: same dis on intervals}
    Let $f : Q \to R$ be an order-preserving map between two posets $Q$ and $R$ that are endowed with metric $d_Q$ and $d_R$. Then, 
    \[
    \dis (\bar f) = \dis (f).
    \]
\end{prop}

\paragraph{Persistence Diagram of a Filtration.}

Let $Q = \{ 0=q_0 < q_1 < \cdots < q_n \} \subseteq [0,\infty)$ be a finite subset containing $0$. $K$ be a finite simplicial complex and let $\Ffunc : Q\to \subcx(K)$ be a filtration. Observe that, for any dimension $d\geq 0$, the inclusion of simplicial complexes $\Ffunc(q_i) \subseteq \Ffunc(q_j)$, for $q_i \leq q_j$, induces canonical inclusions on the cycle and boundary spaces. In particular, for any $i = 0,\ldots,n$, the $d$-th cycle and boundary spaces of $\Ffunc(q_i)$ can be identified with subspaces of $C_d^{\Ffunc(q_n)} = C_d^K$. 

\begin{center}
    \begin{tikzcd}
\Zfunc_d(\Ffunc(q_i)) \arrow[r, hook]                 & \Zfunc_d(\Ffunc(q_j)) \arrow[r, hook]                 & \Zfunc_d(K) \arrow[r, hook] & C_d^K \\
\Bfunc_d(\Ffunc(q_i)) \arrow[r, hook] \arrow[u, hook] & \Bfunc_d(\Ffunc(q_j)) \arrow[r, hook] \arrow[u, hook] & \Bfunc_d(K) \arrow[u, hook] &      
    \end{tikzcd}
\end{center}

\begin{definition}[Birth-death functions, {\cite[Definition 5.6]{mccleary2022edit}}]\label{defn: bd functions}
    Let $\Ffunc : Q \to \subcx(K)$ be a $1$-parameter filtration. For any degree $d\geq 0$, the $d$-th~\emph{birth-death function} associated to $\Ffunc$ is defined as the function $\ZB_d^\Ffunc : \bar Q \to \ZZ$ given by
    \[
    \ZB_d^\Ffunc ([q_i, q_j]) := \Zfunc_d (\Ffunc(q_i))\cap \Bfunc_d (\Ffunc(q_j)).
    \]
\end{definition}

\begin{definition}[Degree-$d$ persistence diagram of a filtration, {\cite[Definition 8.1]{mccleary2022edit}}]
    Let $\Ffunc : Q \to \subcx(K)$ be a $1$-parameter filtration. The \emph{degree-$d$ persistence diagram} of $\Ffunc$, denoted $\PD_q^\Ffunc$ is defined to be the M\"obius inverse of the $d$-th birth-death function of $\Ffunc$. That is,
    \[
    \PD_d^\Ffunc := \partial_{\bar Q} \ZB_d^\Ffunc : \bar Q \to \ZZ
    \]
\end{definition}

For instance, from the (standard) Vietoris-Rips filtration (\Cref{def:VR-filt})  we obtain the (unweighted) persistence diagram  $\PD_d^\VR$.

\begin{remark}
    Note that the birth-death functions in~\cite{mccleary2022edit,orthogonal-mobius} are also defined for 'infinite' intervals of the form $[q_i, \infty)$ to account for topological features that are born at some point and never die. However, since we focus exclusively on the Vietoris-Rips filtration in this paper, every topological feature in degree $d > 0$ will eventually die, and in degree $0$, every persistence diagram will have exactly one interval of the form $[0, \infty)$ with multiplicity 1. Therefore, we may safely disregard the infinite intervals in our setting.
\end{remark}

Note that the objects of the category $\Dgm$ are certain functions whose targets are $\NN$. It is not, a priori, immediately clear that the degree-$d$ persistence diagram of a $1$-parameter filtration takes only non-negative values. We refer to~\cite[Section 9.1]{mccleary2022edit} and~\cite[Proposition 5.10]{gal-conn} for the non-negativity of persistence diagrams of $1$-parameter filtrations.

\begin{prop}[Functoriality of persistence diagrams]\label{prop: functoriality of pd}
For any $d\geq0$, the assignment 
\begin{align*}
    \PD_d : \Fil &\to \Dgm \\
    \Ffunc &\mapsto \PD_d^\Ffunc
\end{align*}
    is a functor. Moreover, 
    \[
    d_{\Dgm}^E (\PD_d^\Ffunc, \PD_d^\Gfunc) \leq d_{\Fil}^E (\Ffunc, \Gfunc)
    \]
    for any $\Ffunc$ and $\Gfunc$ in $\Fil$.
\end{prop}

\begin{proof}
        Let $V$ and $W$ be two finite sets and $K_V$ and $K_W$ be two simplicial complexes over the vertex sets $V$ and $W$ respectively. Let $\Ffunc : Q  \to \subcx (K_V)$ and $\Gfunc : R \to \subcx(K_W)$ be two $1$-parameter filtrations. Let $(f,g, Z, \pi_V, \pi_W)$ be a morphism from $\Ffunc$ to $\Gfunc$. That is,
    \begin{itemize}
        \item $Z$ is a finite set and $V \xtwoheadleftarrow{\pi_V} Z \xtwoheadrightarrow{\pi_W} W$ is a tripod,
        \item $f: Q \leftrightarrows R : g$ is a Galois connection,
        \item $\Gfunc_Z^{\pi_W} (r) = \Ffunc_Z^{\pi_V} (g(r))$ for all $r\in R$.
    \end{itemize}
    Observe that the last condition above implies that
    $\ZB_d^{\Gfunc_Z^{\pi_W}} = \ZB_d^{\Ffunc_Z^{\pi_V}} \circ \bar g$. Then, by RGCT (\cref{thm:rgct}), we obtain
    \[
    PD_d^{\Gfunc_Z^{\pi_W} } = \partial_{\bar R} \ZB_d^{\Gfunc_Z^{\pi_W}} = \partial_{\bar R} \left(\ZB_d^{\Ffunc_Z^{\pi_V}} \circ \bar g \right) = {\bar f}_\sharp \left(\partial_{\bar Q} \ZB_d^{\Ffunc_Z^{\pi_V}}\right) = {\bar f}_\sharp \left( \PD_d^{\Ffunc_Z^{\pi_V}} \right).
    \]
    That is, the Galois connection $f: Q \leftrightarrows R : g$ is a persistence diagram morphism from $\PD_d^{\Ffunc_Z^{\pi_V}}$ to $PD_d^{\Gfunc_Z^{\pi_W} }$. 
    By~\cite[Corollary 4.1]{memoli17}, persistence diagram of a filtration and its pullback filtration are identical on nondiagonal intervals. That is, $\PD_d^\Ffunc = \PD_d^{\Ffunc_Z^{\pi_V}}$ on $\bar Q \setminus \diag (\bar Q)$ and $\PD_d^\Gfunc = \PD_d^{\Gfunc_Z^{\pi_W}}$ on $\bar R \setminus \diag (\bar R)$. Thus, the Galois connection $f: Q \leftrightarrows R : g$ is a persistence diagram morphism from $\PD_d^{\Ffunc}$ to $PD_d^{\Gfunc }$. Hence, $\PD_d$ is a functor.

    The inequality
        \[
    d_{\Dgm}^E (\PD_d^\Ffunc, \PD_d^\Gfunc) \leq d_{\Fil}^E (\Ffunc, \Gfunc)
    \]
    follows from~\cref{prop: general edit distance stability} and~\cref{prop: same dis on intervals} as $\PD_d$ is a functor and $\cost_{\Fil} ((f,g, Z, \pi_V, \pi_W)) = \dis (f) = \dis (\bar f) = \cost_{\Dgm} (f: Q \leftrightarrows R : g)$.
\end{proof}

\begin{proof}[Poof of~\cref{thm: classical edit distance stability of pd}]
    By~\cref{prop: unweighted vr is functor} and~\cref{prop: functoriality of pd}, the result follows.
\end{proof}

\subsection{Weighted Persistence Diagrams}

\begin{definition}[Weighted persistence diagrams]\label{defn: weighted persistence diagrams}
    Let $Q \subseteq [0,\infty)$ be a finite set containing $0$. A \emph{weighted persistence diagram}, written $w\textnormal{-}\sigma : (\bar Q, \mu_{\bar Q}) : \bar Q \to \NN$, is persistence diagram $\sigma : \bar Q \to \NN$ together with a probability measure $\mu_{\bar Q}$ on $\bar Q$ such that $\sigma (I) > 0 $ implies $\mu_{\bar Q} (I) >0$ for every $I \in \bar Q \setminus \diag (\bar Q)$.
\end{definition}

\begin{definition}[Weighted persistence diagram morphisms]
    Let $w\textnormal{-}\sigma : (\bar Q, \mu_{\bar Q}) \to \NN$ and $w\textnormal{-}\tau : (\bar R, \mu_{\bar R}) \to \NN$ be two weighted persistence diagrams. A \emph{weighted persistence diagram morphism} from $w\textnormal{-}\sigma $ to $w\textnormal{-}\tau$ is a Galois connection $
    f : Q \leftrightarrows R : g$ such that 
    \begin{itemize}
        \item $f : Q \leftrightarrows R : g$ is a persistence diagram morphism from $\sigma : Q\to \NN$ to $\tau : R \to \NN$,
        \item $\bar f : (\bar Q, \mu_{\bar Q}) \to (\bar R, \mu_{\bar R})$ is a Monge map.
    \end{itemize}
\end{definition}

\begin{definition}\label{defn: deformation}
    Let $I_1 = [a_1,b_1]$, $I_2 = [a_2,b_2]$, $J_1 = [c_1,d_1]$, and $J_2= [c_2,d_2]$ be intervals in $\RR$. We define the \emph{displacement} associated to $I_1$, $I_2$, $J_1$, and $J_2$ as
    \[
    \defo (I_1, I_2, J_1, J_2) := \max 
    \begin{cases}
        &|(a_1-a_2) - (c_1-c_2)|, \\
        &|(a_1-b_2) - (c_1-d_2)|, \\
        &|(b_1-a_2) - (d_1-c_2)|, \\
        &|(b_1-b_2) - (d_1-d_2)|
    \end{cases}
    \]
\end{definition}
For an interval $I = [a,b] \subseteq \RR$, let $\len(I) := |a-b|$ denote the length of the interval. The displacement $\defo (I_1, I_2, J_1, J_2)$ measures similar is  the pair $(I_1, I_2)$  to the pair $(J_1, J_2)$ by comparing $I_1-I_2 \in \RR^2$ with $J_1 - J_2 \in \RR^2$, $\len(I_1)$ with $\len(J_1)$, and $\len(I_2)$ with $\len(J_2)$. To be precise, one can see that $\defo (I_1, I_2, J_1, J_2) = 0$ if and only if $I_1 - I_2 = J_1 - J_2$ and $\len(I_i) = \len(J_i)$ for $i=1,2$.

The following proposition relates the notion of displacement with the notion of distortion of an order-preserving map.

\begin{prop}\label{prop: meaning of dpl}
    Let $Q, R\subseteq [0,\infty)$ be finite subsets. Let $f : Q \to R$ be an order-preserving map. Then,
    \[
    \dis(f) = \dis (\bar f) = \max_{I_1, I_2 \in \bar Q} \defo (I_1, I_2, \bar f (I_1), \bar f (I_2))
    \]
\end{prop}

\begin{proof}
    We already have the first equality by~\cref{prop: same dis on intervals}. Let $q_1, q_2\in Q$ be such that $\dis(f) = \big| |q_1-q_2| - |f(q_1) - f(q_2)| \big|$. Observe that, since $f$ is order-preserving, we have $$\big | |q_1-q_2| - |f(q_1) - f(q_2)| \big |=  \big | (q_1-q_2) - (f(q_1) - f(q_2)) \big |.$$ Thus, 
    \begin{align*}
        \dis (f) &= \big | |q_1-q_2| - |f(q_1) - f(q_2)| \big | \\
                 &= \big | (q_1-q_2) - (f(q_1) - f(q_2)) \big | \\
                 &= \defo ([q_1, q_1], [q_2, q_2], \bar f ([q_1, q_1]) , \bar f ([q_2, q_2])) \\
                 &\leq \max_{I_1, I_2 \in \bar Q} \defo (I_1, I_2, \bar f (I_1), \bar f (I_2))
    \end{align*}
    Similarly, let $I_1=[q_3, q_4], I_2=[q_5,q_6] \in \bar Q$ be such that 
    \[
    \defo (I_1, I_2, \bar f (I_1), \bar f (I_2)) = \max_{I_1, I_2 \in \bar Q} \defo (I_1, I_2, \bar f (I_1), \bar f (I_2))
    \] 
    Without loss of generality, assume that $\defo (I_1, I_2, \bar f (I_1) \bar f (I_2) = \big|(q_3 - q_5) - (f(q_3) - f(q_5))\big|$. Since $f$ is order-preserving, we have that
    \[
    \big|(q_3 - q_5) - (f(q_3) - f(q_5))\big| = \big||q_3 - q_5| - |f(q_3) - f(q_5)|\big|.
    \]
    Thus, we obtain
    \begin{align*}
        \max_{I_1, I_2 \in \bar Q} \defo (I_1, I_2, \bar f (I_1), \bar f (I_2)) &= \big||q_3 - q_5| - |f(q_3) - f(q_5)|\big| \\
        &\leq \max_{q,q' \in Q} \big | |q-q'| - |f(q)-f(q')|  \big | \\
        &= \dis(f).
    \end{align*}

\end{proof}

Let $\wDgm$ denote the category whose objects are weighted persistence diagrams and whose morphisms are weighted persistence diagram morphisms.

\begin{definition}[p-cost of a weighted persistence diagram morphism]
    Let $ f : Q \leftrightarrows R : g$ be a morphism from a weighted persistence diagram $w\textnormal{-}\sigma : (\bar Q, \mu_{\bar Q}) \to \NN$ to another one $w\textnormal{-}\tau : (\bar R , \mu_{\bar R})\to \NN$. For $p\in [1,\infty)$, we define the $p$-cost of this morphism as
    \[
    \cost_{\wDgm}^p (f : Q \leftrightarrows R : g) := \left ( \sum_{I_1, I_2, \in \bar Q} \defo (I_1, I_2, \bar f (I_1), \bar f(I_2))^p \; \mu_{\bar Q}(I_1) \; \mu_{\bar Q}(I_2)\right )^\frac{1}{p}
    \]
    For $p=\infty$, we define the $\infty$-cost of this morphism as
    \[
    \cost_{\wDgm}^\infty (f : Q \leftrightarrows R : g) := \max_{I_1, I_2 \in \bar Q} \defo (I_1, I_2, \bar f (I_1), \bar f (I_2)).
    \]
\end{definition}

For a fixed $p \in [1,\infty]$, we now have the notion of \emph{$p$-edit distance} as follows. Every morphism in $\wDgm$ has a nonnegative $p$-cost associated with it. As discussed in \cref{subsec: edit}, we obtain an edit distance between objects of $\wDgm$. We refer to this edit distance as the $p$-edit distance and denote it by $d_\wDgm^{E,p}$.

\paragraph{Weighted Persistence Diagram of a Weighted Filtration.}\label{paragraph: w pf of w filtration}
Let $$(Q\subseteq [0,\infty), |\cdot|, \mu_Q, \leq)$$ be a finite linear mm-space. Consider the map $\flip : Q\times Q \to \bar Q$ given by

\begin{equation*}
    \flip(q_1, q_2) :=
    \begin{cases}
        [q_1, q_2] & \text{ if } q_1\leq q_2 \\
        [q_2, q_1] & \text{ if } q_2 < q_1
    \end{cases}
\end{equation*}
Notice that since $Q$ is linearly ordered, the map $\flip$ is defined for every pair $(q_1, q_2) \in Q\times Q$. We define the probability measure $\mu_{\bar Q, \flip}$ on $\bar Q$ as the pushforward $\mu_{\bar Q, \flip}:= \flip_\# (\mu_Q \otimes \mu_Q)$.

\begin{definition}
    Let $w\textnormal{-}\Ffunc : (Q, \mu_Q) \to \subcx(K)$ be a weighted filtration. For $d\geq 0$, we define the degree-$d$ weighted persistence diagram of $w\textnormal{-}\Ffunc$ as the weighted persistence diagram $w\textnormal{-}\PD_d^\Ffunc : (\bar Q, \mu_{\bar Q, \flip} ) \to \NN$.
\end{definition}

For example, from the weighted Vietoris-Rips filtration (\Cref{def:wVR-filt}), we obtain the weighted persistence diagram $w\textnormal{-}\PD_d^\VR$.

\begin{remark}\label{rem: gdd can be captured from wpd}
Note that, for any degree $d$, we can recover the GDD of $X$ from the weighted Vietoris-Rips persistence diagram $w\textnormal{-}\PD_d^\VR(X)$ as follows.
\begin{itemize}
    \item Let $(Q, \mu_Q) := (\ima (d_X), \mu_{GDD(X)})$.
    \item Then, the weight on $w\textnormal{-}\PD_d^\VR(X)$ is given by $\mu_{\bar Q, \flip}:= \flip_\# (\mu_Q \otimes \mu_Q)$.
    \item Then, the measure defined by \[
    \mu_{Q\times Q, \mathsf{unflip}} (q_1,q_2) := \begin{cases}
        \frac{\mu_{\bar Q, \flip}([q_1,q_2])}{2} \;\;\;\;\;\;\;\; &\text{if } q_1 < q_2 \\
        \frac{\mu_{\bar Q, \flip}([q_2,q_1])}{2} \;\;\;\;\;\;\;\; &\text{if } q_2 < q_1 \\
        \mu_{\bar Q, \flip}([q_1,q_2]) \;\;\;\;\; &\text{if } q_1 =q_2
    \end{cases}
    \]
    is precisely the product measure $\mu_Q \otimes \mu_Q$ on $Q\times Q$. That is, $\mu_{Q\times Q, \mathsf{unflip}} = \mu_Q \otimes \mu_Q$.
    \item Thus, $\mu_{GDD(X)} = \mu_Q$ can be recovered as the pushforward of $\mu_{Q\times Q, \mathsf{unflip}}$ under the projection map $\pi_1 : Q\times Q \to Q$. That is, $\mu_{GDD(X)} = (\pi_1)_\# (\mu_{Q\times Q, \mathsf{unflip}})$.
\end{itemize}
    In particular, if two mm-spaces have different global distributions of distances, then their weighted Vietoris-Rips persitence diagrams must differ, where this difference is captured by the weights on these diagrams.
\end{remark}

\begin{figure}
    \centering
    \subfloat[\centering $(X,d_X,\mu_X)$]{{\includegraphics[width=7cm]{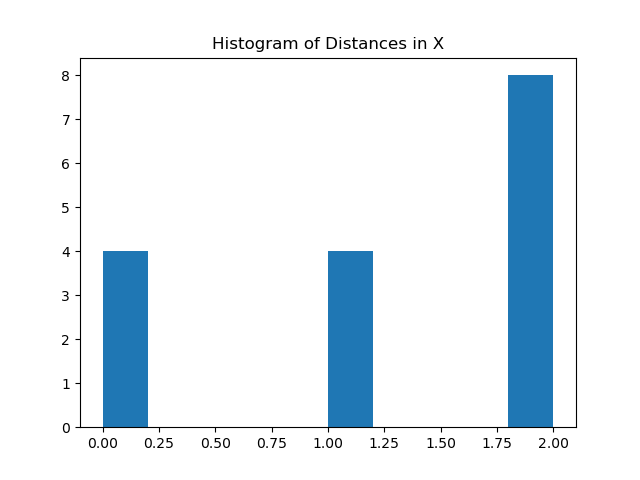} }}%
    \qquad
    \subfloat[\centering $(Y,d_Y,\mu_Y)$]{{\includegraphics[width=7cm]{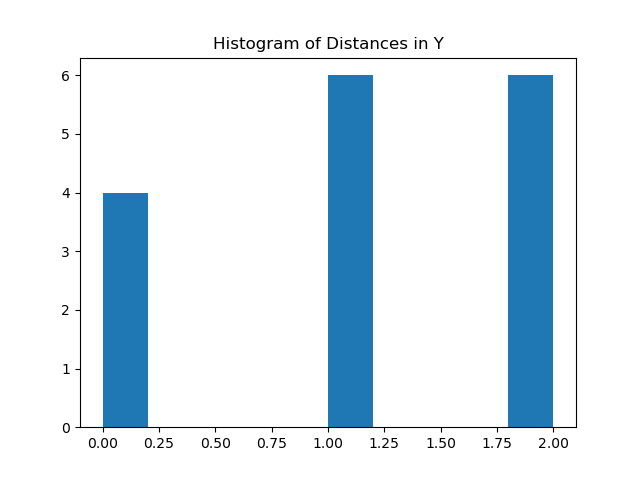} }}%
    \caption{The histograms of distances in $(X,d_X, \mu_X)$ (on the left) and in $(Y,d_Y, \mu_Y)$ (on the right) described in~\cref{ex:ums}.}%
    \label{fig:histograms}%
\end{figure}

\begin{ex}\label{ex:ums}
    Consider the 4-point mm-spaces $(X,d_X, \mu_X)$ and $(Y,d_Y, \mu_Y)$ where $\mu_X$ and $\mu_Y$ are uniform measures on $4$ points and the distance matrices for $d_X$ and $d_Y$ are given by
    \[
    d_X := 
    \begin{bmatrix}
        0 & 1 & 2 & 2 \\
        1 & 0 & 2 & 2 \\
        2 & 2 & 0 & 1 \\
        2 & 2 & 1 & 0
    \end{bmatrix} \;\;\;\;\text{and}  \;\;\;
    d_Y := \begin{bmatrix}
        0 & 1 & 1 & 2 \\
        1 & 0 & 1 & 2 \\
        1 & 1 & 0 & 2 \\
        2 & 2 & 2 & 0
    \end{bmatrix},
    \]
    respectively. Observe that $(X,d_X)$ and $(Y,d_Y)$ are in fact ultrametric spaces and their dendrogram representations are illustrated in~\cref{fig:ums}. As shown in~\cite[Example 3.18]{zhou2024ephemeral}, $(X,d_X)$ and $(Y,d_Y)$ have the same Vietoris-Rips persistence diagrams in all degrees. However, the weighted persistence diagrams of $(X,d_X, \mu_X)$ and $(Y,d_Y, \mu_Y)$ come equipped with different weights as the global distributions of distances of these mm-spaces are different. To be precise, observe that both $\mu_{GDD(X)}$ and $\mu_{GDD(Y)}$ are measures on $\ima(d_X) = \ima(d_Y) = \{ 0<1<2\} =: Q$ with different histograms as depicted in~\cref{fig:histograms}. Thus, as explained in~\cref{rem: gdd can be captured from wpd}, the weighted Vietoris-Rips persistence diagrams of these two mm-spaces must differ.
\end{ex}

\begin{prop}[Functoriality of weighted persistence diagrams]\label{prop: weighted pd are functorial}
    For any $d\geq 0$, the assignment
    \begin{align*}
        w\textnormal{-}\PD_d : \wfil &\to \wDgm \\
                                w\textnormal{-}\Ffunc &\mapsto w\textnormal{-}\PD_d^\Ffunc
    \end{align*}
    is a functor. Moreover, for every $p\in [1,\infty]$
    \[
    d_{\wDgm}^{E,p} (w\textnormal{-}\PD_d^\Ffunc, w\textnormal{-}\PD_d^\Gfunc) \leq 4^{\frac{1} {p}} \cdot d_{\wfil}^{E,p} (w\textnormal{-}\Ffunc, w\textnormal{-}\Gfunc)
    \]
\end{prop}

\begin{proof}
    Let $w\textnormal{-}\Ffunc : (Q,\mu_Q) \to \subcx(K_1)$ and $w\textnormal{-}\Gfunc (R,\mu_R) : \to \subcx (K_2)$ be two weighted filtrations and let $f : Q \leftrightarrows R : g$ be a weighted filtration-preserving morphism. Then, by~\cref{prop: functoriality of pd}, we have that $f : Q \leftrightarrows R : g$ is a persistence diagram morphism from $\PD_d^\Ffunc : \bar Q \to \NN$ to $\PD_d^\Gfunc : \bar R \to \NN$. Thus it remains to verify that $\bar f : (\bar Q, \mu_{\bar Q, \flip}) \to (\bar R, \mu_{\bar R, \flip})$ is a Monge map. The equality $(\bar f)_\# \mu_{\bar Q, \flip} = \mu_{\bar R, \flip}$ follows from the commutative diagram below. 

    \begin{center}
    \begin{tikzcd}
        Q\times Q \arrow[r, "f \times f"] \arrow[d, "\flip"'] & R\times R \arrow[d, "\flip"] \\
        \bar Q \arrow[r, "\bar f"]                                & \bar R                      
    \end{tikzcd}        
    \end{center}

    By~\cref{prop: general edit distance stability}, it suffices to verify that $\cost_{\wDgm}^{p} (f: Q \leftrightarrows R : g) \leq 4^\frac{1}{p} \cdot \dis_p (f)$ in order to show that
    \[
    d_\wDgm^{E,p} (w\textnormal{-}\PD_d^\Ffunc, w\textnormal{-}\PD_d^\Gfunc) \leq 4^\frac{1}{p} \cdot d_\wfil^{E,p} (w\textnormal{-}\Ffunc,w\textnormal{-}\Gfunc).
    \]
    We obtain this inequality, when $p$ is finite, as follows.
    \footnotesize
    \begin{align*}
    &\left (\cost_{\wDgm}^{p} (f: Q \leftrightarrows R : g) \right)^p \\
     &= \sum_{I,J \in \bar Q} \left (\defo(I,J, \bar f (I), \bar f (J))\right)^P \mu_{\bar Q, \flip}(I) \mu_{\bar Q, \flip}(J) \\
    &\leq \sum_{\substack{q_1, q_2, \\ q_3, q_4 \in Q}} \left( \begin{multlined}  \big | ( q_1- q_3) - (f(q_1) - f(q_3)) \big |^p + \big | (q_1 - q_4) - (f(q_1) - f(q_4))\big |^p \\ +\big | (q_2 - q_3) - (f(q_2) - f(q_3)) \big |^p + \big | (q_2 - q_4) - (f(q_2) - f(q_4))\big|^p\end{multlined} \right) \mu_Q (q_1) \mu_Q (q_2) \mu_Q (q_3) \mu_Q (q_4) \\
    &= 4 \cdot \sum_{q,q' \in Q} \big | q-q' - (f(q)-f(q'))\big |^p \mu_Q (q)\mu_Q(q') \\
    &= 4 \cdot \sum_{q,q' \in Q} \big | |q-q'| - |f(q)-f(q')| \big |^p \mu_Q(q) \mu_Q(q') \\
    &= 4 \cdot \left( \dis_p (f) \right)^p
    \end{align*}
\normalsize
When $p=\infty$, we obtain this inequality by~\cref{prop: meaning of dpl} as follows.
\begin{align*}
    \cost_{\wDgm}^{\infty} (f: Q \leftrightarrows R : g)  = \max_{I,J \in \bar Q} \defo(I,J, \bar f (I), \bar f (J)) = \dis (f) = \dis_\infty (f) 
\end{align*}
\end{proof}

\begin{proof}[Proof of~\cref{thm: p stability of weighted persistence diagrams  main thm}]
    By~\cref{prop: weighted pd are functorial} and~\cref{prop: weighted vr is functor}, the result follows.
\end{proof}

\section{OT-like Interpretation of Edit Distance Between (Weighted) Persistence Diagrams}\label{sec: ot-like interpret}

In this section, we show that the edit distance between (weighted) persistence diagrams can be realized in a way that does not require a categorical construction. Although the categorical construction is useful for stability proofs, it does not provide much insight into the interpretability of this distance. By realizing the edit distance between (weighted) persistence diagrams as an OT-like distance, we provide another view for this distance.

\subsection{Interpretation of Edit Distance Between Unweighted Persistence Diagrams}

The edit distance between unweighted persistence diagrams was first defined in~\cite{mccleary2022edit}. In that paper, the authors showed that the edit distance is bi-Lipschitz equivalent to the well-known bottleneck distance~\cite[Theorem 9.1]{mccleary2022edit}. In their proof, constructing a $1$-parameter family of persistence diagrams from a matching between two persistence diagrams was one of the main ingredients. In proofs situated in this section, we will follow a similar strategy of constructing a $1$-parameter family of persistence diagrams from a matching.

\begin{remark}
    In~\cite{mccleary2022edit}, the authors introduce the edit distance between 'signed' persistence diagrams. However, as noted in~\cite[Example 10.1]{edit-arxiv}, this definition can lead to trivial distances between distinct persistence diagrams due to a pathological path. To address this issue, a modification was proposed in~\cite{edit-arxiv}. In this paper, we focus exclusively on $1$-parameter filtrations, which always  yield nonnegative persistence diagrams rather than signed ones. In this case, the nontriviality of the edit distance is simply a consequence of the fact that it is bi-Lipschitz equivalent to the bottleneck distance. Notably, the proof of this equivalence in~\cite[Theorem 9.1]{mccleary2022edit} implicitly assumes nonnegativity.
\end{remark}

\begin{definition}[Matching]
    Let $\sigma : \bar Q \to \NN$ and $\tau : \bar R \to \NN$ be two persistence diagrams. A \emph{matching} between $\sigma$ and $\tau$ is a function $\gamma : \bar Q \times \bar R \to \NN $ such that
    \begin{itemize}
        \item $\sigma (I) = \sum_{J \in \bar R} \gamma (I,J)$ for every $I \in \bar Q \setminus\diag (\bar Q)$,
        \item $\tau (J) = \sum_{I \in \bar Q} \gamma (I,J)$ for every $J \in \bar R \setminus\diag (\bar R)$,
        \item $\gamma([0,0], [0,0]) = 1$.
    \end{itemize}
\end{definition}

In the definition above, the last condition is needed to guarantee that the matching $\gamma$ aligns the two persistence diagrams $\sigma$ and $\tau$ at their common interval $[0,0]$. We now introduce the displacement cost of a matching.

\begin{definition}[Displacement cost of a matching / displacement distance]
    Let $\sigma : \bar Q \to \NN$ and $\tau : \bar R \to \NN$ be two persistence diagrams and let $\gamma : \bar Q \times \bar R \to \NN$ be a matching between them. The \emph{displacement cost} of $\gamma$ is defined to be
    \[
    \defcost (\gamma) := \max_{\substack{I_1, I_2 \in \bar Q \\ J_1, J_2 \in \bar R \\ \gamma(I_1, J_1) >0 \\ \gamma(I_2, J_2) > 0}} \defo (I_1, I_2, J_1, J_2).
    \]
    We define the \emph{displacement distance}, denoted $d_\defo$, between two persistence diagrams $\sigma$ and $\tau$ as the infimum of all the displacement costs of matchings between $\sigma$ and $\tau$. That is,
    \[
    d_\defo (\sigma, \tau) := \inf_{\gamma} \defcost (\gamma).
    \]
\end{definition}

Note that $d_\defo$ satisfies the triangle inequality.

\begin{prop}\label{prop: interpretation unweighted}
    Let $\sigma : \bar Q \to \NN$ and $\tau : \bar R \to \NN$ be two persistence diagrams. Then, 
    \[
    d_{\Dgm}^E (\sigma, \tau) = d_\defo (\sigma,\tau)
    \]
\end{prop}

\begin{proof}
We first show $d_\defo(\sigma, \tau) \leq d_{\Dgm}^E (\sigma, \tau)$. Let
    \[
    \mathcal{P} : \; \; \rho_0 := \sigma \xleftrightarrow{\makebox[1cm]{$F_1$}} \rho_1 \xleftrightarrow{\makebox[1cm]{$F_2$}} \cdots \xleftrightarrow{\makebox[1cm]{$F_{k-1}$}} \rho_{k-1} \xleftrightarrow{\makebox[1cm]{$F_k$}} \tau =: \rho_k
    \]
    be an optimal path between $\sigma$ and $\tau$ in $\Dgm$. That is, each $F_i := (f_i, g_i)$ is a persistence diagram morphism (in either direction), and $d_{\Dgm}^E (\sigma, \tau) = \cost_\Dgm(\mathcal{P}) = \sum_{i=1}^k \dis(f_i).$ Observe that each $F_i$ determines a matching $\gamma_{F_i}$ between $\rho_i$ and $\rho_{i+1}$ with $d_\defcost (\gamma_{F_i}) \leq \dis(f_i)$, see ~\cite[Lemma 9.10]{mccleary2022edit}. Thus, we obtain
    \begin{align*}
        d_\defo (\sigma, \tau) &\leq \sum_{i=0}^{k-1} d_\defo (\rho_i, \rho_{i+1}) \\
                                &\leq \sum_{i=0}^{k-1}  d_{\Dgm}^E  (\rho_i, \rho_{i+1}) \\
                                &= d_{\Dgm}^E (\sigma, \tau).
    \end{align*}    
We now show $d_{\Dgm}^E (\sigma, \tau) \leq d_\defo(\sigma, \tau)$ by adopting the interpolation argument in~\cite[Proposition 9.4]{mccleary2022edit}. Let $\gamma : \bar Q \times \bar R \to \NN$ be a matching between $\sigma$ and $\tau$. For every $t\in [0,1]$, let $A_t := \{ (1-t)I + tJ \mid \gamma (I,J) >0 \}$ and let $S_t := \{\text{endpoints of intervals in } A_t \}$. Note that since $\gamma([0,0],[0,0])>0$, we have that $0\in S_t$, i.e. $S_t$ is nonempty. For every $t\in[0,1]$, we define the function
\begin{align*}
    \nu_t : \overline{S_t} &\to \NN \\
                    K &\mapsto \sum_{\substack{I \in \bar Q, J \in \bar R \\ (1-t)I + tJ = K}} \gamma (I,J).
\end{align*}
Observe that $\nu_0 = \sigma$ and $\nu_1 = \tau$. We will now extract certain persistence diagram morphisms from the family $\{\nu_t \}_{t\in [0,1]}$ in order to obtain a path between $\sigma$ and $\tau$ in $\Dgm$. 

As $t$ varies from $0$ to $1$, there are only finitely many places where the combinatorial
structure of $S_t$ changes. To be precise, we say that a point $t\in [0,1]$ is critical if for all sufficiently small $\delta > 0$, there exists $s\in (t-\delta, t+\delta)$ with $|S_t|\neq |S_s|$. By~\cite[Lemma 9.6]{mccleary2022edit}, if $t\in[0,1]$ is not a critical point, then for any $K \in \supp(\nu_t)$ there is a unique pair of intervals $I \in \bar Q$ and $J \in \bar R$ with $\gamma(I,J)>0$ and $(1-t)I + tJ =K$. This uniqueness allows us to define an order-preserving map, $\alpha_{t,s}$, from $S_t$ to $S_s$, where $t$ is not a critical point and $s\in[0,1]$ is any point with no critical points strictly between $s$ and $t$, as follows. For any $b\in S_t$, there are unique intervals $I\in \bar Q$ and $J \in \bar R$ with $\gamma(I,J)>0$ and either $(1-t)I + tJ$ is equal to $[a,b]$ or $[b,c]$. If $(1-t)I + tJ =[a,b]$, then define $\alpha_{t,s}(b)$ to be the right endpoint of the intervals $(1-s)I + sJ$. Similarly, if $b$ is a left endpoint, then we define $\alpha_{t,s}(b)$ to be the left endpoint of $(1-s)I+sJ$. Finally, we define $\alpha_{s,t}(0) =0$. This map has a right adjoint given by $\beta_{t,s}(x) = \max (b \in S_t \mid \alpha_{t,s}(b) \leq x)$. Then, by~\cite[Lemma 9.8]{mccleary2022edit}, the Galois connection $(\alpha_{t,s}, \beta_{t,s})$ determines a persistence diagram morphism.

We now show that the cost of the persistence diagram morphism $(\alpha_{t,s}, \beta_{t,s})$ is upper bounded by $|t-s|\cdot \defcost(\gamma)$. By~\cref{prop: same dis on intervals}, we have that 
\[
\cost_{\Dgm}(\alpha_{t,s} : S_t \leftrightarrows S_s : \beta_{t,s}) = \dis (\alpha_{t,s}) = \max_{u,v\in S_t} |(u-v) - (\alpha_{t,s}(u) - \alpha_{t,s}(v))|.
\]
Let $u,v\in S_t$. There are four cases: $u$ and $v$ could be a left endpoint or a right endpoint of an interval in $A_t$. We will analyze only one case as all the cases are similar. Assume that $u$ is the left endpoint of the interval $(1-t)I_1 + t J_1$ with $\gamma(I_1, J_1)>0$ and $I_1 = [a_1,b_1]$, $J_1 = [c_1,d_1]$. Assume that $v$ is the right endpoint of $(1-t)I_2 + tJ_2$ with $\gamma(I_2, J_2)>0$ and $I_2 = [a_2,b_2]$, $J_2 = [c_2,d_2]$. Then, we have that $u = (1-t)a_1 + tc_1$ and $v = (1-t)b_2 + td_2$. Then,
\begin{align*}
    &|(u-v) - (\alpha_{t,s}(u) - \alpha_{t,s}(v))| \\
    &= \Big|\Big((1-t)a_1 +tc_1 - \big( (1-t)b_2 + td_2 \big) \Big) - \Big((1-s)a_1 +sc_1 - \big( (1-s)b_2 + sd_2 \big) \Big)\Big| \\
    &=\big| (s-t)(a_1-b_2) - (s-t)(c_1-d_2) \big| \\
    &= |s-t|\cdot \big| (a_1-b_2) - (c_1-d_2) \big| \\
    &\leq |s-t|\cdot \max 
    \begin{cases}
        &|(a_1-a_2) - (c_1-c_2)|, \\
        &|(a_1-b_2) - (c_1-d_2)|, \\
        &|(b_1-a_2) - (d_1-c_2)|, \\
        &|(b_1-b_2) - (d_1-d_2)|
    \end{cases} \\
    &= |s-t|\cdot \defo(I_1, I_2, J_1, J_2) \\
    &\leq |s-t| \max_{\substack{{I'}_1, {I'}_2 \in \bar Q \\ {J'}_1, {J'}_2 \in \bar R \\ \gamma({I'}_1, {J'}_1) >0 \\ \gamma({I'}_2, {J'}_2) > 0}} \defo ({I'}_1, {I'}_2, {J'}_1, {J'}_2) \\
    &\leq |s-t|\cdot \defcost(\gamma).
\end{align*}

Thus, we obtain
\[
\cost_{\Dgm}(\alpha_{t,s} : S_t \leftrightarrows S_s : \beta_{t,s})= \max_{u,v\in S_t} |(u-v) - (\alpha_{t,s}(u) - \alpha_{t,s}(v))| \leq |s-t|\cdot \defcost(\gamma)
\]

Now, let $\{ 0 = s_0 < \cdots < s_N =1 \} \subseteq [0,1]$ be the set of critical points and choose $\{ t_0 < \cdots <t_{N-1}\} \subseteq [0,1]$ with $0=s_0 < t_0 < s_1 <\cdots <t_{N-1} < s_N =1$. Then, the Galois connections $(\alpha_{t_i, s_i}, \beta_{t_i, s_i})$ and $(\alpha_{t_i, s_{i+1}}, \beta_{t_i, s_{i+1}})$ form a path $\mathcal{P}$ between $\sigma$ and $\tau$ in $\Dgm$ with 
\[
\cost_{\Dgm}(\mathcal{P}) \leq \sum_{i=0}^{N-1} ((t_i -s_i) + (s_{i+1}-t_i)) \cdot \defcost(\gamma) = \defcost(\gamma).
\]
Therefore, $d_{\Dgm}^E (\sigma, \tau) \leq d_\defo(\sigma, \tau)$.
\end{proof}

\subsection{Interpretation of the~\texorpdfstring{$p$}--Edit Distance Between Weighted Persistence Diagrams}

Analogously to the previous section, we now re-express the $p$-edit distance between weighted persistence diagrams in an OT-like fashion.

\begin{definition}[Weighted Matching]
    Let $w\textnormal{-}\sigma : (\bar Q, \mu_{\bar Q}) \to \NN$ and $w\textnormal{-}\tau : (\bar R, \mu_{\bar R}) \to \NN$ be two weighted persistence diagrams. A \emph{weighted matching} between $w\textnormal{-}\sigma$ and $w\textnormal{-}\tau$ is a pair $(\gamma, \eta)$ where $\gamma$ is a matching between $\sigma$ and $\tau$ and $\eta$ is a coupling between $\mu_{\bar Q}$ and $\mu_{\bar R}$ such that
    \[
    \gamma (I,J) > 0 \implies \eta(I,J) > 0.
    \]
\end{definition}

\begin{definition}[$p$-displacement cost of a matching / $p$-displacement distance]
    Let $w\textnormal{-}\sigma : (\bar Q,\mu_{\bar Q}) \to \NN$ and $w\textnormal{-}\tau : (\bar R, \mu_{\bar R}) \to \NN$ be weighted two persistence diagrams and let $(\gamma, \eta)$ be a weighted matching between them. For $p\in [1,\infty)$, the \emph{$p$-displacement cost} of $(\gamma, \eta)$ is defined to be
    \[
    \defcost_p (\gamma,\eta) := \left ( \sum_{\substack{I_1, I_2 \in \bar Q \\ J_1, J_2 \in \bar R}} \left( \defo (I_1, I_2, J_1, J_2)\right)^p \eta (I_1, J_1) \eta (I_2, J_2) \right)^\frac{1}{p}.
    \]
    For $p=\infty$, the \emph{$\infty$-displacement cost} of $(\gamma, \eta)$ is defined to be
    \[
    \defcost_p (\gamma,\eta) := \max_{\substack{\eta(I_1, J_1) >0 \\ \eta (I_2, J_2) > 0}} \defo(I_1, I_2, J_1, J_2)
    \]
    We define the \emph{$p$-displacement distance}, denoted $d_{\defo,p}$, between two weighted persistence diagrams $w\textnormal{-}\sigma$ and $w\textnormal{-}\tau$ as the infimum of all the $p$-displacement costs of weighted matchings between $w\textnormal{-}\sigma$ and $w\textnormal{-}\tau$. That is,
    \[
    d_{\defo,p} (w\textnormal{-}\sigma, w\textnormal{-}\tau) := \inf_{(\gamma,\eta)} \defcost_p (\gamma,\eta).
    \]
\end{definition}

\begin{prop}\label{prop: OT formulation of dEp}
    Let $w\textnormal{-}\sigma : (\bar Q,\mu_{\bar Q}) \to \NN$ and $w\textnormal{-}\tau : (\bar R,\mu_{\bar R}) \to \NN$ be two weighted persistence diagrams. Then, 
    \[
    d_{\wDgm}^{E,p} (w\textnormal{-}\sigma, w\textnormal{-}\tau) = d_{\defo,p} (w\textnormal{-}\sigma,w\textnormal{-}\tau)
    \]
\end{prop}

\begin{proof}
We first show $d_{\defo,p} (w\textnormal{-}\sigma,w\textnormal{-}\tau) \leq d_{\wDgm}^{E,p} (w\textnormal{-}\sigma, w\textnormal{-}\tau)$. Let $f : Q \leftrightarrows R : g$ be a weighted persistence diagram morphism from $w\textnormal{-}\sigma$ to $w\textnormal{-}\tau$. Then, $f : Q \leftrightarrows R : g$ induces a weighted matching $(\gamma_f, \eta_f)$ as follows.
$$\gamma_f (I,J) := \begin{cases}
    \sigma (I) &\text{ if } \bar f(I) = J \\
    0 &\text{ otherwise.}
\end{cases}$$
And, the coupling between $\mu_{\bar Q}$ and $\mu_{\bar R}$ is given by $\left(id_{\bar Q}, \bar f\right)_\# (\mu_{\bar Q})$. This induced weighted matching $(\gamma_f, \eta_f)$ satisfies
\[
\defcost_p (\gamma_f,\eta_f) \leq \cost_{\wDgm}^p (f : Q \leftrightarrows R : g).
\]
Thus, we obtain $d_{\defo,p} (w\textnormal{-}\sigma,w\textnormal{-}\tau) \leq d_{\wDgm}^{E,p} (w\textnormal{-}\sigma, w\textnormal{-}\tau)$.

We now show $d_{\wDgm}^{E,p} (w\textnormal{-}\sigma, w\textnormal{-}\tau) \leq d_{\defo,p} (w\textnormal{-}\sigma,w\textnormal{-}\tau)$ by adopting the interpolation argument in~\cite[Proposition 9.4]{mccleary2022edit} and in the proof of~\cref{prop: interpretation unweighted}. Let $(\gamma, \eta)$ be a weighted matching between $w\textnormal{-}\sigma$ and $w\textnormal{-}\tau$. For every $t\in [0,1]$, let $\mathfrak{A}_t := \{ (1-t)I + tJ \mid \eta (I,J) >0 \}$ and let $\mathfrak{S}_t := \{\text{endpoints of intervals in } \mathfrak{A}_t \}$. Note that since $\gamma([0,0],[0,0])>0$, we have that $0\in \mathfrak{S}_t$, i.e., $\mathfrak{S}_t$ is nonempty. For every $t\in[0,1]$, we define

\begin{align*}
    \nu_t : \overline{\mathfrak{S}_t} &\to \NN \\
                    K &\mapsto \sum_{\substack{I \in \bar Q, J \in \bar R \\ (1-t)I + tJ = K}} \gamma (I,J),\\
    \mu_t : \overline{\mathfrak{S}_t} &\to [0,1] \\
                    K &\mapsto \sum_{\substack{I \in \bar Q, J \in \bar R \\ (1-t)I + tJ = K}} \eta (I,J).
\end{align*}
Observe that $ \nu_t : ( \overline{\mathfrak{S}_t}, \mu_t) \to \NN$ is a weighted persistence diagram. As in the proof of~\cref{prop: interpretation unweighted}, there are only finitely many places where the combinatorial
structure of $\mathfrak{S}_t$ changes as $t$ varies from 0 to 1. Similarly to the proof of~\cref{prop: interpretation unweighted}, we set critical points to be the values $t\in [0,1]$ such that for all sufficiently small $\delta > 0$, there exists $s \in (t-\delta, t+\delta)$ with $|\mathfrak{S}_t|\neq|\mathfrak{S}_s|$. As constructed in proof of~\cref{prop: interpretation unweighted}, we construct the Galois connections $\alpha_{t,s} : \mathfrak{S}_t \leftrightarrows \mathfrak{S}_s : \beta_{t,s}$ where $t$ is a non-critical point and $s$ is any point with no critical points strictly between $s$ and $t$. These Galois connections $(\alpha_{t,s}, \beta_{t,s})$ determine weighted persistence diagram morphisms from the weighted persistence diagram $\nu_t$ to $\nu_s$ with
\[
\cost_{\wDgm}^p (\alpha_{t,s} : \mathfrak{S}_t \leftrightarrows \mathfrak{S}_s : \beta_{t,s}) \leq |s-t|\cdot \defcost_p (\gamma,\eta).
\]

Therefore, we obtain a path between $w\textnormal{-}\sigma$ and $ w\textnormal{-}\tau$ whose $p$-cost is at most $\defcost_p (\gamma,\eta)$. Thus, we obtain $d_{\wDgm}^{E,p} (w\textnormal{-}\sigma, w\textnormal{-}\tau) \leq d_{\defo,p} (w\textnormal{-}\sigma,w\textnormal{-}\tau)$.
\end{proof}

\begin{corollary}
For two finite mm-spaces $(X,d_X,\mu_X)$ and $(Y,d_Y, \mu_Y)$, we have
    \[
    d_{\defo,p} \left(w\textnormal{-}\PD_d^{\VR(X)},w\textnormal{-}\PD_d^{\VR(Y)}\right) \leq 4^{\frac{p+1}{p}} \cdot \dgw{p} ((X,d_X, \mu_X), (Y,d_Y,\mu_Y)).
    \]
\end{corollary}

\section{Discriminating Power of Weighted Persistence Diagrams}\label{sec: discriminating power}

Weighed persistence diagrams blend two notions: classical persistence diagrams and global distributions of distances. In this section, we show that  weighted persistence diagrams have stronger discriminating power than both  classical persistence diagrams and global distributions of distances. We will investigate two cases separately: $p=\infty$ and $p \in [1,\infty)$. 

\subsection{Case \texorpdfstring{$p=\infty$}{}}

We focus on the case $p=\infty$ separately. We first prove that  classical Vietoris-Rips persistence diagrams  are $\infty$-Gromov-Wasserstein stable.

\begin{prop}\label{prop: infinity gw stability of pd}
    Let $(X,d_X,\mu_X)$ and $(Y,d_Y,\mu_Y)$ be two finite mm-spaces. Then,
    \[
    d_\Dgm^E \left (\PD_d^{\VR(X)}, \PD_d^{\VR(Y)} \right) \leq 4\cdot \dgw{\infty} ((X,d_X, \mu_X), (Y,d_Y,\mu_Y)).
    \]
\end{prop}

\begin{proof}
    By~\cref{prop: unweighted vr is functor} and~\cref{prop: functoriality of pd}, we have that 
    \[
    d_\Dgm^E \left (\PD_d^{\VR(X)}, \PD_d^{\VR(Y)} \right) \leq 4\cdot \dgh ((X,d_X), (Y,d_Y)).
    \]
    Combining this inequality with the fact,~\cite[Theorem 5.2 (b)]{facundo-gw}, that $$\dgh ((X,d_X), (Y,d_Y)) \leq \dgw{\infty} ((X,d_X,\mu_X), (Y,d_Y,\mu_Y)),$$ we obtain
    \[
    d_\Dgm^E \left (\PD_d^{\VR(X)}, \PD_d^{\VR(Y)} \right) \leq 4\cdot \dgw{\infty} ((X,d_X, \mu_X), (Y,d_Y,\mu_Y)).
    \]
\end{proof}

\begin{remark}
    Notice that the inequality $$d_\Dgm^E \left (\PD_d^{\VR(X)}, \PD_d^{\VR(Y)} \right) \leq 4\cdot \dgh ((X,d_X), (Y,d_Y))$$ in the proof above can also be obtained from the the following facts
    \begin{enumerate}
        \item The edit distance is bi-Lipschitz equivalent to bottleneck distance,~\cite[Theorem 9.1]{mccleary2022edit}.
        \[
        d_B \left(\PD_d^{\VR(X)}, \PD_d^{\VR(Y)}\right) \leq d_\Dgm^E \left (\PD_d^{\VR(X)}, \PD_d^{\VR(Y)} \right) \leq 2\cdot d_B \left(\PD_d^{\VR(X)}, \PD_d^{\VR(Y)}\right)
        \]
        \item Persistence diagrams are Gromov-Hausdorff stable,~\cite[Theorem 3.1]{chazal2009}.
        \[
        d_B \left(\PD_d^{\VR(X)}, \PD_d^{\VR(Y)}\right) \leq 2 \cdot \dgh(X,Y).
        \]
    \end{enumerate}
\end{remark}

We now describe how~\cref{thm: p stability of weighted persistence diagrams  main thm}, permits improving upon the $\infty$-Gromov-Wasserstein stability of persistence diagrams given in~\cref{prop: infinity gw stability of pd}.

\begin{prop}\label{prop: p infty improved stability}
    Let $(X,d_X,\mu_X)$ and $(Y,d_Y,\mu_Y)$ be two finite mm-spaces. Then,
    \begin{align*}
        d_\Dgm^E \left (\PD_d^{\VR(X)}, \PD_d^{\VR(Y)} \right) &\leq 
        d_{\wDgm}^{E,\infty} \left(w\textnormal{-}\PD_d^{\VR(X)}, w\textnormal{-}\PD_d^{\VR(Y)}\right) \\ &\leq 4 \cdot \dgw{\infty} ((X,d_X,\mu_X)), (Y,d_Y,\mu_Y)).
    \end{align*}
    
\end{prop}

\begin{proof}
    The assignment 
    \begin{align*}
        \wDgm &\to \Dgm \\
         w\textnormal{-}\sigma &\mapsto \sigma
    \end{align*}
    is a functor. Moreover, for a morphism $f: P\leftrightarrows Q : g$ between two objects $w\textnormal{-}\sigma : (\bar Q, \mu_{\bar Q}) \to \ZZ$ and $w\textnormal{-}\tau : (\bar R, \mu_{\bar R}) \to \ZZ$ in $\wDgm$, we have that 
    \[
    \cost_{\Dgm} (f: P\leftrightarrows Q : g) \leq \cost_{\wDgm}^\infty (f: P\leftrightarrows Q : g).
    \]
    Thus, by~\cref{prop: general edit distance stability}, we have that 
    \[
        d_\Dgm^E \left (\sigma, \tau \right) \leq 
        d_{\wDgm}^{E,\infty} \left(w\textnormal{-}\sigma, w\textnormal{-}\tau\right)
    \]
    for any $w\textnormal{-}\sigma $ and $ w\textnormal{-}\tau$ in $\wDgm$.
\end{proof}

\begin{remark}
    Note that in~\cref{prop: p infty improved stability}, we use the edit distance to compare persistence diagrams, even though the bottleneck distance is more commonly employed for such comparisons. However, these two distances are shown to be bi-Lipschitz equivalent~\cite[Theorem 9.1]{mccleary2022edit}, implying that they offer the same discriminative power. Crucially, the edit distance is structurally compatible with the $\infty$-edit distance used to compare weighted persistence diagrams, as both are defined via cost functions on morphisms within their respective categories. This structural alignment makes the edit distance particularly well-suited for our goal of comparing the discriminative power of weighted versus unweighted persistence diagrams.
\end{remark}

We now present an example in which $d_\Dgm^E$ (equivalently, the bottleneck distance $d_B$) fails to distinguish two mm-spaces based on their persistence diagrams, whereas the weighted bottleneck distance $d_{\wDgm}^{E,\infty}$ succeeds in capturing the difference between their weighted persistence diagrams, thereby distinguishing the mm-spaces.

\begin{ex}\label{ex: same pd-1 different wpd-1}
    In~\cref{fig: same pd different wpd}, we illustrate two mm-spaces $(X,d_X,\mu_X)$ and $(Y,d_Y,\mu_Y)$. The set $X$ consists of vertices of a regular hexagon with diameter $2$, $d_X$ is the restriction of the Euclidean distance, and $\mu_X$ is the uniform measure. The set $Y$ contains the set $X$ and also contains a single middle point of two consecutive vertices in $X$, $d_Y$ is the restriction of the Euclidean distance, and $\mu_Y$ is the uniform measure. In degree-$1$, the persistence diagrams of the Vietoris-Rips filtrations of $(X,d_X)$ and $(Y,d_Y)$ are the same, see~\cref{fig: same pd 1}. Thus, 
    \[
    d_\Dgm^E \left (\PD_1^{\VR(X)}, \PD_1^{\VR(Y)} \right) = 0.
    \]
    On the other hand, we have
    \[
    d_{\wDgm}^{E,\infty} \left(w\textnormal{-}\PD_1^{\VR(X)}, w\textnormal{-}\PD_1^{\VR(Y)}\right) >0
    \]
    since the weights on $w\textnormal{-}\PD_1^{\VR(X)}$ and $w\textnormal{-}\PD_1^{\VR(Y)}$ have different supports. 
\end{ex}

\begin{figure}
    \centering
    \subfloat[\centering $(X,d_X,\mu_X)$]{{\includegraphics[width=7cm]{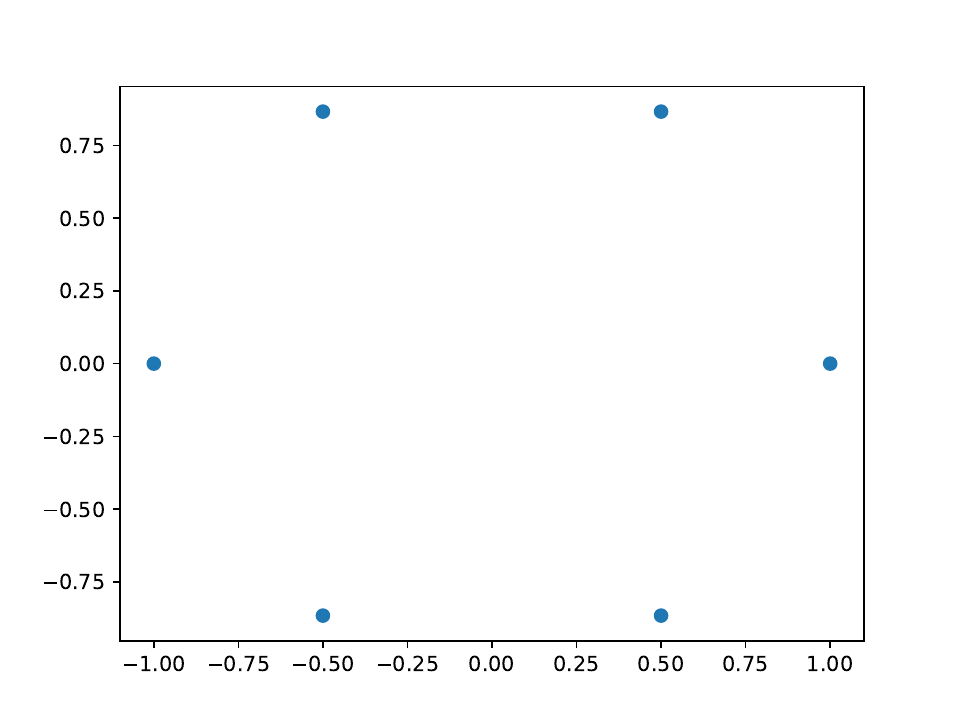} }}%
    \qquad
    \subfloat[\centering $(Y,d_Y,\mu_Y)$]{{\includegraphics[width=7cm]{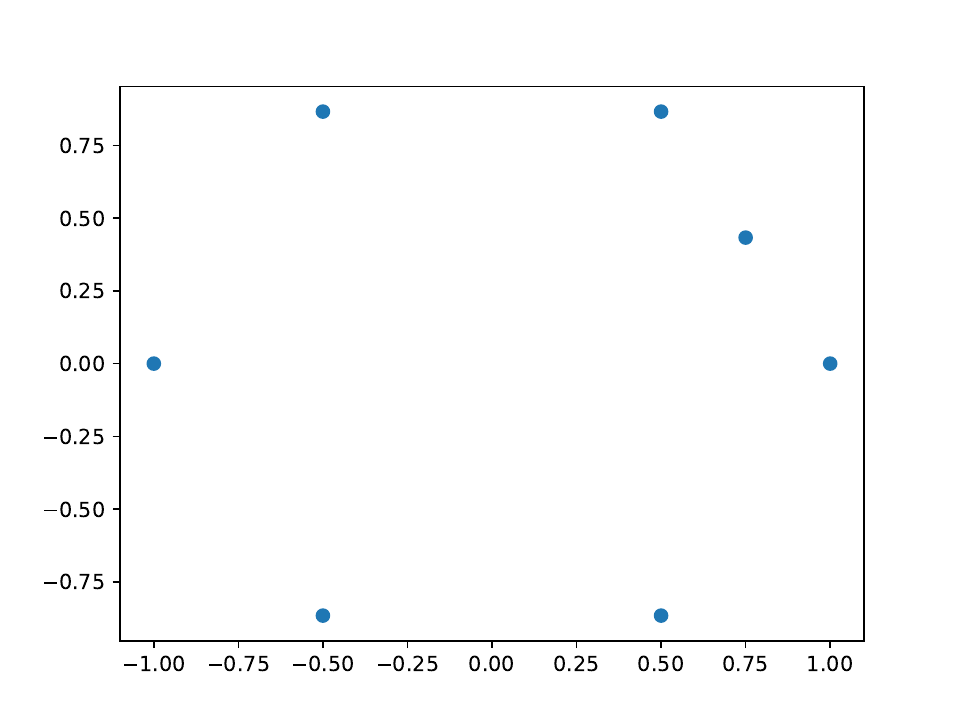} }}%
    \caption{Two mm-spaces with the same degree-1 persistence diagrams but with different weighted persistence diagrams.}%
    \label{fig: same pd different wpd}%
\end{figure}

\begin{figure}
    \centering
    \includegraphics{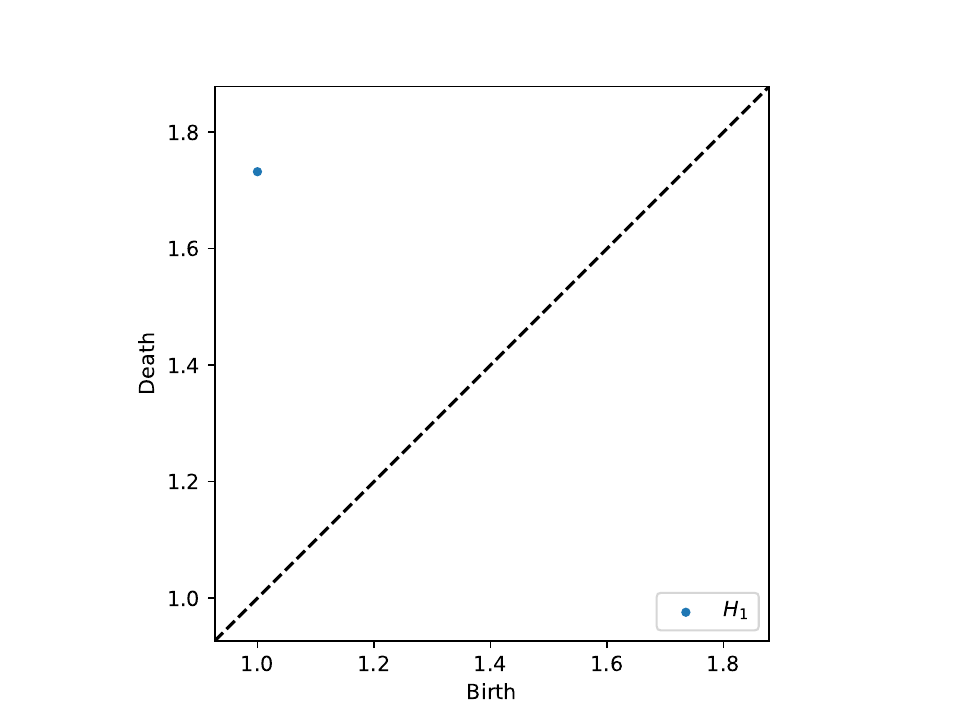}
    \caption{Degree-1 persistence diagrams of the Vietoris-Rips filtrations of $(X,d_X)$ and $(Y,d_Y)$ described in~\cref{ex: same pd-1 different wpd-1}.}
    \label{fig: same pd 1}
\end{figure}

We now pay special attention to the other component of the weighted persistence diagrams: the global distribution of distances. Let $(X,d_X,\mu_X)$ be a finite mm-space and let $\mu_{GDD(X)} = (d_X)_\# (\mu_X \otimes \mu_X)$ be the global distribution of distances of $X$ as in~\cref{defn: global dist of dist}. When constructing a weighted persistence diagram from a finite mm-space $(X,d_X,\mu_X)$, we record a probability measure on the set of intervals of $D_X = \ima(d_X)$. This probability measure is given by $\mu_{\overline{D_X}} := \flip_\# \left(\mu_{GDD(X)} \otimes \mu_{GDD(X)}\right)$ and called the \emph{flipped global distribution of distances}. Recall that $\flip : D_X \times D_X \to \overline{D_X}$ is the function described in~\cref{paragraph: w pf of w filtration} and given by

\begin{equation*}
    \flip_{D_X}(d_1, d_2) :=
    \begin{cases}
        [d_1, d_2] & \text{ if } d_1\leq d_2 \\
        [d_2, d_1] & \text{ if } d_2 < d_1.
    \end{cases}
\end{equation*}

We now present that flipped global distributions of distances are $\infty$-Gromov-Wasserstein stable.

\begin{prop}\label{prop: gw infinity stability of flipped}
    Let $(X,d_X,\mu_X)$ and $(Y,d_Y,\mu_Y)$ be two finite mm-spaces. Then,
    \[
    \dgw{\infty} \left ( \left(\overline{D_X}, ||\cdot||_\infty, \mu_{\overline{D_X}}\right), \left( \overline{D_Y}, ||\cdot||_\infty, \mu_{\overline{D_Y}} \right)\right) \leq 4 \cdot \dgw{\infty} \left ( (X,d_X,\mu_X), (Y,d_Y,\mu_Y) \right).
    \]
    
\end{prop}

\begin{proof}
    By~\cite[Theorem 5.1 (c) and Proposition 6.2]{facundo-gw}, we have that
    \[
    2\cdot \dgw{\infty}\left( \left(D_X, |\cdot|, \mu_{GDD(X)}\right), \left(D_Y, |\cdot|, \mu_{GDD(Y)}\right) \right) \leq 4 \cdot \dgw{\infty} \left ( (X,d_X,\mu_X), (Y,d_Y,\mu_Y) \right).
    \]
    Observe that a coupling $\nu$ between $\mu_{GDD(X)}$ and $\mu_{GDD(Y)}$ determines a coupling $$\overline \nu := \left(\flip_{D_X} \times \flip_{D_Y}\right)_\# (\nu \otimes \nu)$$ between $\mu_{\overline{D_X}}$ and $\mu_{\overline{D_Y}}$, where $\left(\flip_{D_X} \times \flip_{D_Y}\right) : D_X \times D_Y \times D_X \times D_Y \to \overline{D_X} \times \overline{D_Y}$ is the function whose restriction to $D_X\times D_X$ is $\flip_{D_X}$ and whose restriction to $D_Y\times D_Y$ is $\flip_{D_Y}$. The following computation shows that $\dis_{\infty} (\bar \nu) \leq 2\cdot \dis_{\infty} (\nu)$.
    \begin{align*}
        \dis_{\infty} (\bar \nu) &= \max_{\substack{\bar \nu (I_1, J_1) >0 \\ \bar \nu (I_2, J_2) >0}} \big | ||I_1 - I_2||_\infty - ||J_1 - J_2||_\infty \big | \\
        &= \max_{\substack{\nu \otimes \nu (r_1, s_1, r_2, s_2) >0 \\ \nu \otimes \nu (r_3, s_3, r_4, s_4) >0}} \big | \max(|r_1 - r_3|, |r_2 - r_4|) - \max (|s_1 - s_3|-|s_2 - s_4|) \big | \\
        &\leq \max_{\substack{\nu \otimes \nu (r_1, s_1, r_2, s_2) >0 \\ \nu \otimes \nu (r_3, s_3, r_4, s_4) >0}} \max \big(|r_1 - r_3|-|s_1-s_3|, |r_2-r_4|-|s_2-s_4|\big). \\
        &\leq 2\cdot \max_{\substack{\nu(r,s) >0 \\ \nu(t,u)>0}} \big | |r-t| - |s-u| \big | \\
        &= 2\cdot \dis_{\infty} (\nu).
    \end{align*}
    Hence, we obtain that 
    \begin{gather*}
        \dgw{\infty} \left ( \left(\overline{D_X}, ||\cdot||_\infty, \mu_{\overline{D_X}}\right), \left( \overline{D_Y}, ||\cdot||_\infty, \mu_{\overline{D_Y}} \right)\right) \\ \leq \\ 2\cdot \dgw{\infty}\left( \left(D_X, |\cdot|, \mu_{GDD(X)}\right), \left(D_Y, |\cdot|, \mu_{GDD(Y)}\right) \right) \\ \leq \\ 4 \cdot \dgw{\infty} \left ( (X,d_X,\mu_X), (Y,d_Y,\mu_Y) \right).
    \end{gather*}
    
\end{proof}

Recall that our result,~\cref{thm: p stability of weighted persistence diagrams  main thm}, states that
\[
d_{\wDgm}^{E,\infty} \left(w\textnormal{-}\PD_d^{\VR(X)}, w\textnormal{-}\PD_d^{\VR(Y)}\right) \leq 4 \cdot \dgw{\infty} ((X,d_X,\mu_X)), (Y,d_Y,\mu_Y)).
\]

We now present that our result,~\cref{thm: p stability of weighted persistence diagrams  main thm}, improves upon the $\infty$-Gromov-Wasserstein stability of flipped global distributions of distances,~\cref{prop: gw infinity stability of flipped}.

\begin{prop}\label{prop: p infty gdd improved stability}
    Let $(X,d_X,\mu_X)$ and $(Y,d_Y,\mu_Y)$ be two finite mm-spaces. Then,
    \begin{align*}
        \dgw{\infty} \left ( \left(\overline{D_X}, ||\cdot||_\infty, \mu_{\overline{D_X}}\right), \left( \overline{D_Y}, ||\cdot||_\infty, \mu_{\overline{D_Y}} \right)\right) &\leq 
        d_{\wDgm}^{E,\infty} \left(w\textnormal{-}\PD_d^{\VR(X)}, w\textnormal{-}\PD_d^{\VR(Y)}\right) \\ &\leq 4 \cdot \dgw{\infty} ((X,d_X,\mu_X)), (Y,d_Y,\mu_Y)).
    \end{align*}
    
\end{prop}

\begin{proof}
    By~\cref{prop: OT formulation of dEp}, we have that 
    \[
    d_{\wDgm}^{E,\infty} \left(w\textnormal{-}\PD_d^{\VR(X)}, w\textnormal{-}\PD_d^{\VR(Y)}\right) = d_{\defo,\infty} \left(w\textnormal{-}\PD_d^{\VR(X)}, w\textnormal{-}\PD_d^{\VR(Y)}\right)
    \]
    Let $(\gamma, \eta)$ be a weighted matching between the weighted persistence diagrams $w\textnormal{-}\PD_d^{\VR(X)}$ and $w\textnormal{-}\PD_d^{\VR(Y)}$. Then, $\eta$ is a coupling between $\mu_{\overline{D_X}}$ and $\mu_{\overline{D_Y}}$ since the weights on the weighted persistence diagrams $w\textnormal{-}\PD_d^{\VR(X)}$ and $w\textnormal{-}\PD_d^{\VR(Y)}$ are given by $\mu_{\overline{D_X}}$ and $\mu_{\overline{D_Y}}$. The following computation shows that $\dis_{\infty} (\eta) \leq \defcost_{\infty} (\gamma, \eta)$.
    \begin{align*}
        \dis_{\infty} (\eta) &= \max_{\substack{\eta(I_1, J_1)>0 \\ \eta(I_2, J_2)>0}} \big | ||I_1-I_2||_\infty - ||J_1 - J_2||_\infty \big | \\
        &= \max_{\substack{\eta([a_1, b_1], [c_1, d_1])>0 \\ \eta([a_2, b_2], [c_2,d_2])>0}} \big | \max\left(|a_1-a_2|, |b_1-b_2|\right) - \max(|c_1-c_2|, |d_1-d_2|) \big | \\
        &\leq \max_{\substack{\eta([a_1, b_1], [c_1, d_1])>0 \\ \eta([a_2, b_2], [c_2,d_2])>0}} \max\left(\big ||a_1-a_2|-|c_1-c_2|\big |, \big| |b_1-b_2|-|d_1-d_2| \big |\right) \\
        &\leq \max_{\substack{\eta([a_1, b_1], [c_1, d_1])>0 \\ \eta([a_2, b_2], [c_2,d_2])>0}} \defo ([a_1,b_1], [a_2, b_2], [c_1, d_1], [c_2, d_2]) \\
        &= \defcost_{\infty} (\gamma, \eta).
    \end{align*}
    Hence, we conclude that 
    \begin{align*}
         \dgw{\infty} \left ( \left(\overline{D_X}, ||\cdot||_\infty, \mu_{\overline{D_X}}\right), \left( \overline{D_Y}, ||\cdot||_\infty, \mu_{\overline{D_Y}} \right)\right) &\leq d_{\defo,\infty} \left(w\textnormal{-}\PD_d^{\VR(X)}, w\textnormal{-}\PD_d^{\VR(Y)}\right) \\
         &= d_{\wDgm}^{E,\infty} \left(w\textnormal{-}\PD_d^{\VR(X)}, w\textnormal{-}\PD_d^{\VR(Y)}\right) \\
         &\leq 4 \cdot \dgw{\infty} ((X,d_X,\mu_X)), (Y,d_Y,\mu_Y)).
    \end{align*}
\end{proof}

We now present an example where two mm-spaces have the same (flipped) global distributions of distances, thus global distributions fail to discriminate them. On the other hand, $d_{\wDgm}^{E,\infty}$ captures the difference between the weighted persistence diagrams, hence discriminating the mm-spaces.

\begin{ex}\label{ex: same gdd but different wpd}
    In~\cref{fig: same gdd different wpd}, we illustrate two mm-spaces $(X,d_X,\mu_X)$ and $(Y,d_Y,\mu_Y)$. The set $X = \{ (0,0), (1,1), (3,1), (4,0) \} \subseteq \RR^2$ is endowed with the Euclidean distance and $\mu_X$ is the uniform measure. Similarly, the set $Y = \{ (0,0), (3,1), (3,-1), (4,0) \} \subseteq \RR^2$ is endowed with the Euclidean distance and $\mu_Y$ is the uniform measure. These two mm-spaces $(X,d_X,\mu_X)$ and $(Y,d_Y, \mu_Y)$ have the same (flipped) global distributions of distances, as stated in~\cite[Section 2.1 and Figure 4]{Boutin2004}. Thus,
    \[
    \dgw{\infty} \left ( \left(\overline{D_X}, ||\cdot||_\infty, \mu_{\overline{D_X}}\right), \left( \overline{D_Y}, ||\cdot||_\infty, \mu_{\overline{D_Y}} \right)\right) = 0.
    \]
    On the other hand, we have
    \[
    d_{\wDgm}^{E,\infty} \left(w\textnormal{-}\PD_0^{\VR(X)}, w\textnormal{-}\PD_0^{\VR(Y)}\right) >0
    \]
    since the unweighted persistence diagrams $\PD_0^{\VR(X)}$ and $\PD_0^{\VR(Y)}$ have different supports. 
\end{ex}

\begin{figure}
    \centering
    \subfloat[\centering $(X,d_X,\mu_X)$]{{\includegraphics[width=7cm]{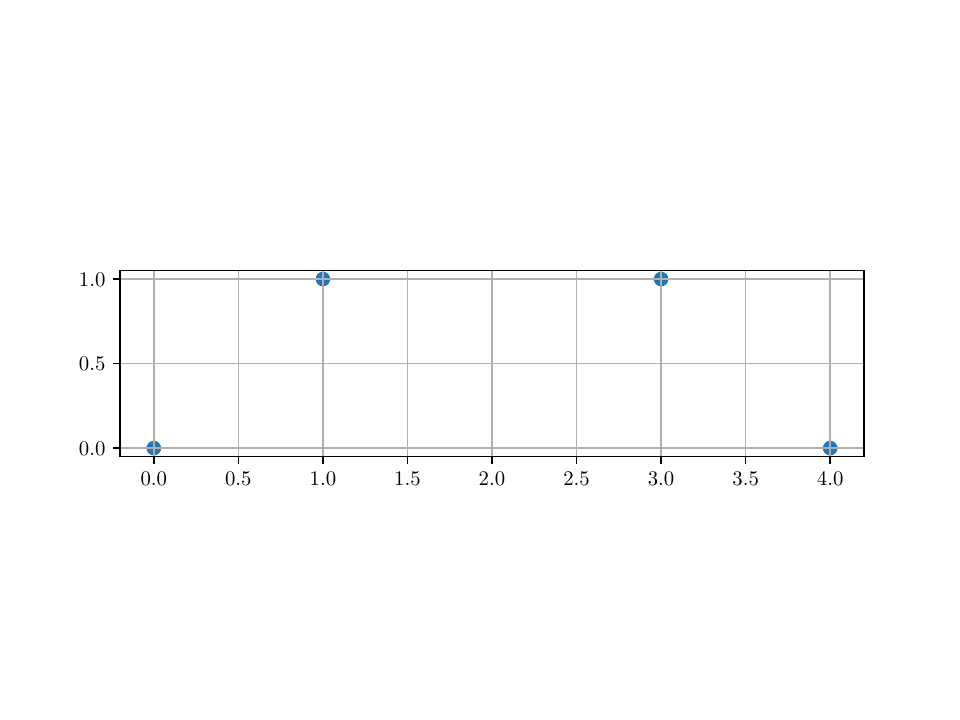} }}%
    \qquad
    \subfloat[\centering $(Y,d_Y,\mu_Y)$]{{\includegraphics[width=7cm]{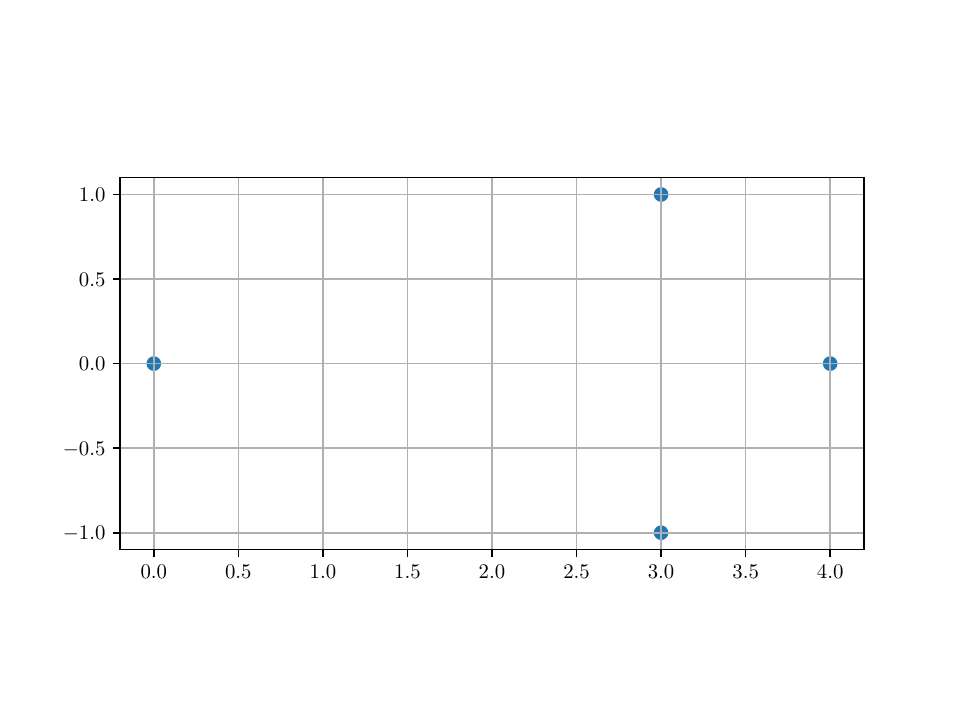} }}%
    \caption{Two mm-spaces with the same (flipped) global distributions of distances but with different weighted persistence diagrams.}%
    \label{fig: same gdd different wpd}%
\end{figure}

\begin{figure}
    \centering
    \subfloat[\centering $\PD_0^{\VR(X)}$]{{\includegraphics[width=7cm]{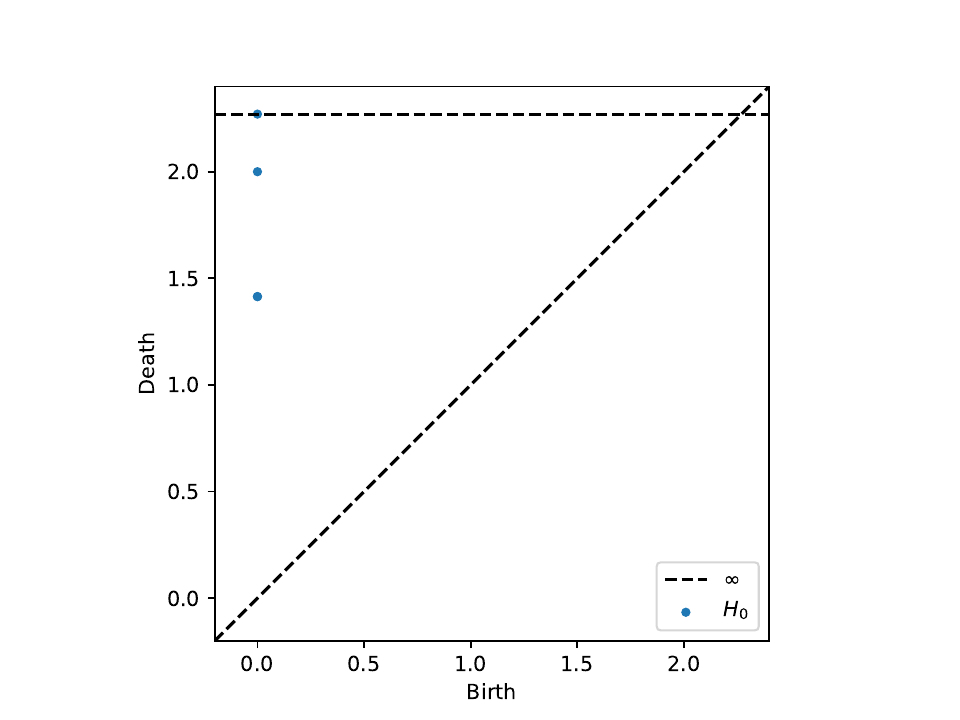} }}%
    \qquad
    \subfloat[\centering $\PD_0^{\VR(Y)}$]{{\includegraphics[width=7cm]{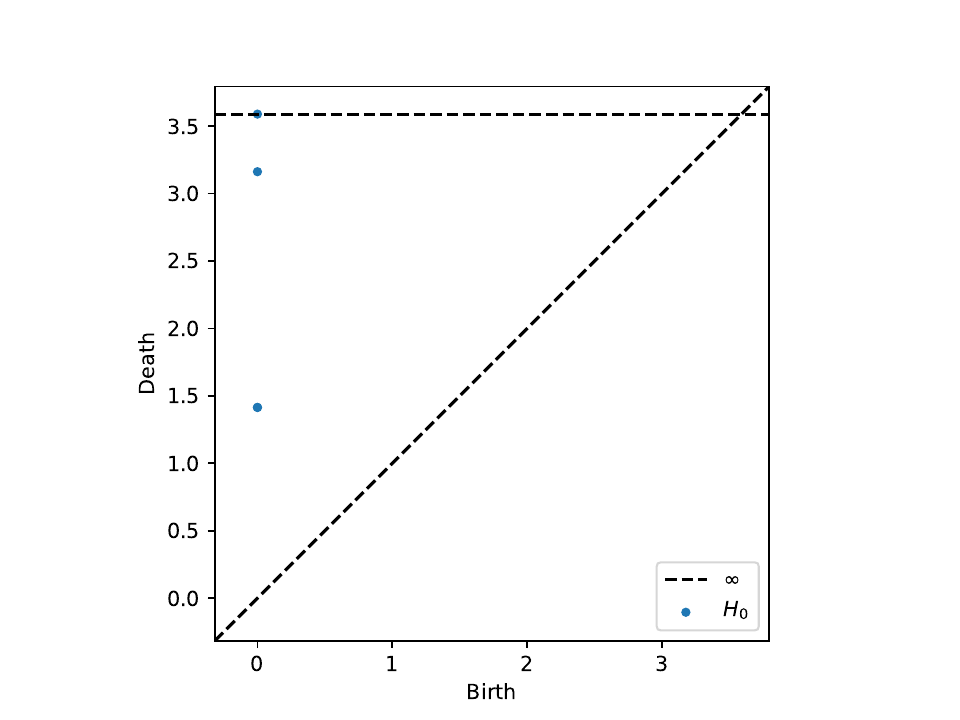} }}%
    \caption{Degree-0 persistence diagrams of the Vietoris-Rips filtrations of $(X,d_X)$ and $(Y,d_Y)$ described in~\cref{ex: same gdd but different wpd} and depicted in~\cref{fig: same gdd different wpd}.}%
    \label{fig: different pds h0}%
\end{figure}

\subsection{Case \texorpdfstring{$p \in [1,\infty)$}{}}

For finite values of $p$, the $p$-edit distance between weighted persistence diagrams becomes solely dependent on the weights, disregarding the topological information captured by the diagrams themselves. We make this precise in the following.

\begin{prop}\label{prop: independence from diagram finite p}
    $w\textnormal{-}\sigma_0 : (\bar Q, \mu_0) \to \NN$ and $w\textnormal{-}\sigma_1 : (\bar Q, \mu_1) \to \NN$ be two weighted persistence diagrams with $\mu_0 = \mu_1 (=: \mu)$. Then, $d_{\wDgm}^{E,p} \left (w\textnormal{-}\sigma_0, w\textnormal{-}\sigma_1  \right) = 0$ for $p\in [1,\infty)$.
\end{prop}

\begin{proof}
    Let $\eps>0$. We will construct a path between $w\textnormal{-}\sigma_0$ and $w\textnormal{-}\sigma_1$ whose cost is controlled by $\eps$. Without loss of generality, we can assume that there exists $I_0 \neq I_1 \in \bar Q$ such that $\sigma_0(I_0) = \sigma_1 (I_1) = 1$, $\sigma_0(I_1) = \sigma_1(I_0) = 0$, and for all $I \in \bar Q \setminus \{ I_0, I_1 \}$ $\sigma_0(I) = \sigma_1 (I)$. For $t\in [0,1]$ let $I_t = (1-t)I_0 + t I_1$. Consider the following weighted persistence diagrams.
    \begin{enumerate}
        \item $w\textnormal{-}\sigma_{\frac{1}{4}}$ whose weights $\mu_{\frac{1}{4}}$ are given by $\mu_{\frac{1}{4}}\left(I_{\frac{1}{4}}\right) = \mu_{\frac{1}{4}} \left(I_{\frac{3}{4}}\right) = \eps $, $\mu_{\frac{1}{4}}(I_0) =\mu(I_0) - \eps$, $\mu_{\frac{1}{4}}(I_1) = \mu (I_1) - \eps$, $\mu_{\frac{1}{4}} (I) = \mu(I)$ for every other interval $I$ and whose diagram is given by $\sigma_{\frac{1}{4}}\left(I_{\frac{1}{4}}\right) =1$, $\sigma_{\frac{1}{4}}(I_0) = \sigma_{\frac{1}{4}} \left(I_{\frac{3}{4}}\right) = \sigma_{\frac{1}{4}} (I_1) = 0$, $\sigma_{\frac{1}{4}} (I) = \mu(I)$ for every other interval $I$. See the figure below for a visual illustration.
        \begin{center}
            \includegraphics[scale=14]{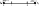}
        \end{center}
        \item $w\textnormal{-}\sigma_{\frac{2}{4}}$ whose weights and diagram is illustrated below.
        \begin{center}
            \includegraphics[scale=14]{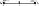}
        \end{center}
        \item $w\textnormal{-}\sigma_{\frac{3}{4}}$ whose weights and diagram is illustrated below.
        \begin{center}
            \includegraphics[scale=14]{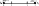}
        \end{center}        
        Then, we obtained a path $$\mathcal{P} : w\textnormal{-}\sigma_0 \leftarrow w\textnormal{-}\sigma_{\frac{1}{4}} \to w\textnormal{-}\sigma_{\frac{2}{4}} \leftarrow w\textnormal{-}\sigma_{\frac{3}{4}} \to w\textnormal{-}\sigma_1 $$ with $\cost_{\wDgm}^p (\mathcal{P}) \leq \eps^{\frac{1}{p}} \cdot M$ for some $M>0$ that only depends on $p$ and $Q$. Hence, by letting $\eps$ go to zero, we obtained that  $d_{\wDgm}^{E,p} \left (w\textnormal{-}\sigma_0, w\textnormal{-}\sigma_1  \right) = 0$.
    \end{enumerate}
\end{proof}

We now present our result which is analogous to~\cref{prop: p infty gdd improved stability}. 

\begin{prop}\label{prop: finite p gdd vs wpd comparison}
    Let $(X,d_X,\mu_X)$ and $(Y,d_Y,\mu_Y)$ be two finite mm-spaces and let $p\in [1,\infty)$. Then,
    \begin{align*}
        \dgw{p} \left ( \left(\overline{D_X}, ||\cdot||_\infty, \mu_{\overline{D_X}}\right), \left( \overline{D_Y}, ||\cdot||_\infty, \mu_{\overline{D_Y}} \right)\right) &\leq 
        d_{\wDgm}^{E,p} \left(w\textnormal{-}\PD_d^{\VR(X)}, w\textnormal{-}\PD_d^{\VR(Y)}\right) \\ &\leq 4^{\frac{p+1}{p}} \cdot \dgw{p} ((X,d_X,\mu_X)), (Y,d_Y,\mu_Y)).
    \end{align*}    
\end{prop}

\begin{proof}
        By~\cref{prop: OT formulation of dEp}, we have that 
    \[
    d_{\wDgm}^{E,p} \left(w\textnormal{-}\PD_d^{\VR(X)}, w\textnormal{-}\PD_d^{\VR(Y)}\right) = d_{\defo,p} \left(w\textnormal{-}\PD_d^{\VR(X)}, w\textnormal{-}\PD_d^{\VR(Y)}\right)
    \]
    Let $(\gamma, \eta)$ be a weighted matching between the weighted persistence diagrams $w\textnormal{-}\PD_d^{\VR(X)}$ and $w\textnormal{-}\PD_d^{\VR(Y)}$. Then, $\eta$ is a coupling between $\mu_{\overline{D_X}}$ and $\mu_{\overline{D_Y}}$ since the weights on the weighted persistence diagrams $w\textnormal{-}\PD_d^{\VR(X)}$ and $w\textnormal{-}\PD_d^{\VR(Y)}$ are given by $\mu_{\overline{D_X}}$ and $\mu_{\overline{D_Y}}$. The following computation shows that $\dis_{p} (\eta) \leq \defcost_{p} (\gamma, \eta)$.
    \begin{align*}
        (\dis_{p} (\eta))^p &= \sum_{I_1, I_2, J_1, J_2} \big | ||I_1-I_2||_\infty - ||J_1 - J_2||_\infty \big |^p \; \eta(I_1, J_1) \; \eta(I_2,J_2) \\
        &= \sum_{\substack{[a_1, b_1], [a_2, b_2] \\ [c_1, d_1], [c_2,d_2]}} \big | \max\left(|a_1-a_2|, |b_1-b_2|\right) - \max(|c_1-c_2|, |d_1-d_2|) \big |^p \; \eta(I_1, J_1) \; \eta(I_2,J_2) \\
        &\leq \sum_{\substack{[a_1, b_1], [a_2, b_2] \\ [c_1, d_1], [c_2,d_2]}} \max\left(\big ||a_1-a_2|-|c_1-c_2|\big |, \big| |b_1-b_2|-|d_1-d_2| \big |\right)^p \; \eta(I_1, J_1) \; \eta(I_2,J_2) \\
        &\leq \sum_{\substack{[a_1, b_1], [a_2, b_2] \\ [c_1, d_1], [c_2,d_2]}} \defo ([a_1,b_1], [a_2, b_2], [c_1, d_1], [c_2, d_2])^p \; \eta(I_1, J_1) \; \eta(I_2,J_2) \\
        &= (\defcost_{\infty} (\gamma, \eta))^p.
    \end{align*}
    Hence, we conclude that 
    \begin{align*}
         \dgw{p} \left ( \left(\overline{D_X}, ||\cdot||_\infty, \mu_{\overline{D_X}}\right), \left( \overline{D_Y}, ||\cdot||_\infty, \mu_{\overline{D_Y}} \right)\right) &\leq d_{\defo,p} \left(w\textnormal{-}\PD_d^{\VR(X)}, w\textnormal{-}\PD_d^{\VR(Y)}\right) \\
         &= d_{\wDgm}^{E,p} \left(w\textnormal{-}\PD_d^{\VR(X)}, w\textnormal{-}\PD_d^{\VR(Y)}\right) \\
         &\leq 4 \cdot \dgw{p} ((X,d_X,\mu_X)), (Y,d_Y,\mu_Y)),
    \end{align*}
    where the last inequality follows from~\cref{thm: p stability of weighted persistence diagrams  main thm}.
\end{proof}

As noted earlier in~\cref{prop: independence from diagram finite p}, $d_{\wDgm}^{E,p}$ is insensitive to the diagrams and it only compares the weights. In the case of weighted persistence diagrams of weighted Vietoris-Rips filtrations of finite mm-spaces, these weights are the flipped global distributions. When $(X,d_X,\mu_X)$ and $(Y,d_Y,\mu_Y)$ are finite mm-spaces, both \\ $\dgw{p} \left ( \left(\overline{D_X}, ||\cdot||_\infty, \mu_{\overline{D_X}}\right), \left( \overline{D_Y}, ||\cdot||_\infty, \mu_{\overline{D_Y}} \right)\right)$ and $d_{\wDgm}^{E,p} \left(w\textnormal{-}\PD_d^{\VR(X)}, w\textnormal{-}\PD_d^{\VR(Y)}\right)$ only depends on the flipped global distributions of distances. So, one might suspect that the inequality
\[
\dgw{p} \left ( \left(\overline{D_X}, ||\cdot||_\infty, \mu_{\overline{D_X}}\right), \left( \overline{D_Y}, ||\cdot||_\infty, \mu_{\overline{D_Y}} \right)\right) \leq 
        d_{\wDgm}^{E,p} \left(w\textnormal{-}\PD_d^{\VR(X)}, w\textnormal{-}\PD_d^{\VR(Y)}\right)
\]
that we present in~\cref{prop: finite p gdd vs wpd comparison} could always be an equality. However, this is not the case. While $\dgw{p}$ utilizes distortion to assign a cost to a coupling, $d_{\wDgm}^{E,p}$ utilizes the displacement as we defined in~\cref{defn: deformation}. Unlike distortion, which remains unaffected by the order of intervals, displacement is highly sensitive to the order. We now present an example where $$\dgw{p} \left ( \left(\overline{D_X}, ||\cdot||_\infty, \mu_{\overline{D_X}}\right), \left( \overline{D_Y}, ||\cdot||_\infty, \mu_{\overline{D_Y}} \right)\right) = 0$$ but $$d_{\wDgm}^{E,p} \left(w\textnormal{-}\PD_d^{\VR(X)}, w\textnormal{-}\PD_d^{\VR(Y)}\right) > 0.$$

\begin{ex}\label{ex: same gdd but displacement distinguishes}
    Consider the two mm-spaces $(X,d_X,\mu_X)$ and $(Y,d_Y,\mu_Y)$ depicted in~\cref{fig: triangle and line}. We illustrate their flipped global distributions of distances in~\cref{fig: gdd triangle and line}. For $p\in[1,\infty)$ We have that
    \[
    \dgw{p} \left ( \left(\overline{D_X}, ||\cdot||_\infty, \mu_{\overline{D_X}}\right), \left( \overline{D_Y}, ||\cdot||_\infty, \mu_{\overline{D_Y}} \right)\right) = 0.
    \]
    On the other hand, we have that
    \[
    d_{\wDgm}^{E,p} \left(w\textnormal{-}\PD_d^{\VR(X)}, w\textnormal{-}\PD_d^{\VR(Y)}\right) > 0.
    \]
\end{ex}

\begin{figure}
    \centering
    \subfloat[\centering $(X,d_X,\mu_X)$]{{\includegraphics[width=5cm]{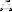} }}%
    \qquad
    \subfloat[\centering $(Y,d_Y,\mu_Y)$]{{\includegraphics[width=5cm]{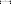} }}%
    \caption{Two mm-spaces whose flipped global distributions of distances cannot be discriminated by $\dgw{p}$ but can be discriminated by $d_{\wDgm}^{E,p}$.}%
    \label{fig: triangle and line}%
\end{figure}

\begin{figure}
    \centering
    \subfloat[\centering $\left( \overline{D_X}, ||\cdot||_\infty, \mu_{\overline{D_X}} \right)$]{{\includegraphics[width=7cm]{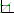} }}%
    \qquad
    \subfloat[\centering $\left( \overline{D_Y}, ||\cdot||_\infty, \mu_{\overline{D_Y}} \right)$]{{\includegraphics[width=7cm]{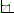} }}%
    \caption{Flipped global distributions of distances of the mm-spaces $(X,d_X,\mu_X)$ and $(Y,d_Y,\mu_Y)$ depicted in~\cref{fig: triangle and line}}%
    \label{fig: gdd triangle and line}%
\end{figure}

\clearpage

\bibliographystyle{alpha}
\bibliography{references}{}

\end{document}